\documentclass[11pt,letter]{article}

\usepackage{bbold,bbm}

\usepackage{latexsym, amsmath, amssymb}
\usepackage{fontenc, indentfirst}
\usepackage{multicol}
\usepackage[usenames,dvipsnames]{xcolor}

\usepackage{marginnote}
\usepackage{enumitem}

\usepackage[normalem]{ulem}
\usepackage{cancel}
  
\usepackage{theorem}
\usepackage{stmaryrd}

\usepackage{graphics,graphicx}

\newtheorem{theorem}{Theorem}[section]
\newtheorem{lemma}[theorem]{Lemma}
\newtheorem{proposition}[theorem]{Proposition}
\newtheorem{definition}[theorem]{Definition}
\newtheorem{corollary}[theorem]{Corollary}
{\theorembodyfont{\rmfamily}
\newtheorem{remark}[theorem]{Remark}
}

\DeclareRobustCommand{\lara}[1]{\langle#1\/\rangle}

\newlength{\mywidth}

\newcommand{\N}{\mathbb{N}}
\newcommand{\Z}{\mathbb{Z}}
\newcommand{\R}{\mathbb{R}}
\newcommand{\T}{\mathbb{T}}
\newcommand{\E}{\mathcal{E}}
\newcommand{\1}{\mathbb{1}}

\newcommand{\Om}{\Omega}

\newcommand{\Id}{I\!d}

\newcommand{\supp}{{\rm supp}}
\newcommand{\Vol}{{\rm Vol}}
\renewcommand{\S}{\mathcal{S}}

\newcommand{\black}[1]{{\color{black}#1}}

\newcommand{\red}[1]{{\color{black}#1}}

\def\qed{\hfill $\square$ \goodbreak \medskip}
\def\eps{\varepsilon}

\newcounter{mnotecount}[section]

\newcommand{\rmnote}[1]{}

\begin{document}
\title{Polygons as maximizers of Dirichlet energy or first eigenvalue of Dirichlet-Laplacian among convex planar domains}
\author{Jimmy {\sc Lamboley}\footnote{Institut de Math\'ematiques de Jussieu-PRG, Sorbonne Universit\'e
4 place Jussieu, 75252 Paris Cedex 05, France},  
Arian {\sc Novruzi}\footnote{Department of Mathematics and Statistics, University of Ottawa, STEM Complex,  Ottawa, Ontario, Canada},
Michel {\sc Pierre}\footnote{\'Ecole Normale Sup\'erieure de Rennes and Institut de Recherche Math\'ematique de Rennes,
Campus de Ker Lann, 35170-Bruz, France}
}
\maketitle

\begin{abstract}
{
{We prove that solutions to several shape optimization problems in the
plane, with a convexity constraint on the admissible domains, are polygons. The main
terms of the shape functionals we consider are either the Dirichlet energy $E_{f}(\Om)$ of the
Laplacian in the domain $\Om$ or the first eigenvalue $\lambda_{1}(\Om)$ of the Dirichlet-Laplacian.
Usually, one considers minimization of such functionals (often with measure constraint),
as for example for the famous Saint-Venant and Faber-Krahn inequalities. By adding
the convexity constraint (and possibly other natural constraints), we instead consider the rather unusual and difficult question of maximizing these functionals. This paper follows a series of papers by the authors, where the leading idea is that a certain concavity property of the shape functional that is minimized leads optimal shapes
to locally saturate their convexity constraint, which geometrically means that they are
polygonal. In these previous papers, the leading term in the shape functional
was usually the opposite of the perimeter, for which the aforementioned concavity property was
rather easy to obtain through computations of its second order shape derivative.
By carrying classical shape calculus, a similar concavity property can be observed for
the opposite of $E_{f}(\Om)$ or $\lambda_{1}(\Om)$ when shapes are smooth and convex. The main novelty in
the present paper is the proof of a weak convexity property of $E_{f}(\Om)$ and $\lambda_{1}(\Om)$
among planar convex shapes, namely rather non-smooth shapes. This involves new computations and estimates of the
second order shape derivatives of $E_{f}(\Om)$ and $\lambda_{1}(\Om)$ interesting for themselves.}}
\end{abstract}

\section{Introduction}
\paragraph{The problems:}
In this paper, we will discuss the four following shape optimization problems:
\begin{equation}
\max\Big\{E_{f}(\Om)+\mu|\Om|,\; \Om\subset\R^2,\; \textrm{ convex},\; \Om\in\S_{ad}\Big\},
\label{e:Pb1}
\end{equation}
\begin{equation}
\max\Big\{\lambda_{1}(\Om)+\mu|\Om|,\; \Om\subset\R^2,\; \textrm{ convex},\; \Om\in\S_{ad}\Big\},
\label{e:Pb2}
\end{equation}
\noindent 
{which will be referred to as penalized problems,} and
\begin{equation}
\max\Big\{E_{f}(\Om),\; \Om\subset\R^2,\; \textrm{ convex},\; \Om\in\S_{ad}\textrm{ and }|\Om|=m_{0}\Big\},
\label{e:Pb3}
\end{equation}
\begin{equation}
\max\Big\{\lambda_{1}(\Om),\; \Om\subset\R^2,\; \textrm{ convex},\; \Om\in\S_{ad}\textrm{ and }|\Om|=m_{0}\Big\},
\label{e:Pb4}
\end{equation}
\noindent 
{which are the volume constrained versions.}

Here, for a given $f\in L^2(\R^2)$ and $\Om\subset\R^2$ bounded and open, we denoted
\begin{eqnarray}
E_{f}(\Om)
&=&
\min\left\{\frac{1}{2}\int_{\Om}|\nabla {U}|^2-\int_{\Om}f{U}, \;\;{U}\in H^1_{0}(\Om)\right\},
\;\; \textrm{resp.}
\label{e:Ef}
\\
\lambda_{1}(\Om)
&=&
\min\left\{\frac{\int_{\Om}|\nabla {U}|^2}{\int_{\Om}{U}^2}, \;\;{U}\in H^1_{0}(\Om)\setminus\{0\}\right\},
\label{e:l1}
\end{eqnarray}
the Dirichlet energy (associated to the right hand side $f$), resp. the first Dirichlet eigenvalue of the Laplace operator { of the domain $\Om$}. 
Also, {$|\Om|=\Vol(\Om)$ denotes the volume of the shape $\Om$}, $\mu$, $m_0$ are positive real numbers,
and $\S_{ad}$ denotes a set of planar shapes including various constraints different from the convexity or the volume constraint. We will consider the case
\begin{equation}\label{e:Sad}
\S_{ad}=\{
\Om\subset \R^2,\;\Om\;\textrm{open}, \;\;\;D_{1}\subset \Om\subset D_{2}\},
\end{equation}
where $D_{1},D_{2}$ are fixed convex open sets{, even though our results could be adapted to other cases, see \cite[Remark 4.6]{LNP2}}.

It is well known that there exist $U$ and $U_{1}$ solutions to \eqref{e:Ef} and \eqref{e:l1} respectively, satisfying
\begin{equation}\label{eq:state}
\left\{\begin{array}{l}-\Delta U=f\textrm{ in }\Om\\[3mm] U\in H^1_{0}(\Om)\end{array}\right.\quad\textrm{ and }\quad
\left\{\begin{array}{l}-\Delta U_{1}=\lambda_{1}(\Om)U_{1}\textrm{ in }\Om\\[3mm]U_{1}\in H^1_{0}(\Om)\end{array}\right.
\end{equation}
{In the case of $U_{1}$, we can also make the choice that $U_{1}>0$ and $\int_{\Om}U_{1}^2dx =1$, which we will do in the rest of the paper.}

\paragraph*{{Framework and }motivation.}
This paper continues a series of three papers by the authors \cite{LN,LNP1,LNP2} (see also \cite[Section 3.5]{H20Sha}) and solves the main open question that was left open in these papers (\cite[Open problem 3.58]{H20Sha}). {The focus of \cite{LN,LNP1,LNP2}}  is the analysis of optimal shapes under convexity constraint. It was mainly observed that if one considers ${\Om}_0$ a solution of
$$\min\Big\{J(\Om), \;\;\Om\subset\R^2,\; \textrm{ convex},\; \Om\in\S_{ad}\Big\},$$
and if $J$ satisfies the {following {\it  concavity property}\footnote{{We call this property a "concavity property" because it implies $J''({\Om_0})({V,V})<0$ if $v$ has small support.}} (see \cite[Theorem 3]{LNP1}), 
\begin{equation}\label{e:suit-conv-prop}
J''(\Omega_0)({V,V})
{\leq}
-\alpha|v|_{H^\sigma(\partial\Omega_0)}^2 + 
\beta\|v\|_{H^s(\partial\Omega_0)}^2+
\gamma|v|_{H^{\sigma}(\partial\Omega_0)}\|v\|_{H^s(\partial\Omega_0)},
\end{equation}
with $0\leq s<\sigma=1$, $\alpha>0$, $\beta,\gamma\geq0$, 
${V}{(x)}=v{(x)}\frac{x}{|x|}$ { with } ${v\in W^{1,\infty}}(\partial\Omega_0)$,
}
then the set of admissible deformations of ${\Om}_0$ is of finite dimension (see also \cite[Theorem 4.8]{LNP2} for a similar result in dimension higher than 2).
In the case \eqref{e:Sad}, geometrically this means that the free boundary {$\partial{\Om}_0\cap (D_{2}\setminus{\overline{D_{1}}})$} is polygonal. A similar behavior was previously noticed for a problem linked to the famous Newton's problem of minimal resistance, see \cite{LP01New}, which was a big inspiration for our initial results.

{The aforementioned {\it concavity property} was formulated through second order shape derivatives of the shape functional. 
Let us quickly clarify the framework we consider in the whole paper (see \cite{HP2018} for more details):
we consider 
\begin{equation}\label{Omega}
\Om
 \;  {\rm open,\;  bounded},
\end{equation}
 Given a  shape functional $E(\cdot) :\mathcal{S}_{ad}\to \R$ defined on a family $\mathcal{S}_{ad}$ of admissible subsets of $\R^2$, we consider
$$\mathcal{E}:{W^{1,\infty}(\R^2,\R^2)}\to\R, \;\;
{\E(V):=E((\Id+V)(\Om))},$$
where $\Id:x\mapsto x$ { and $W^{1,\infty}(\R^2,\R^2)$ is endowed with its usual norm}.
Then, $E$ is said to be twice shape differentiable at $\Omega$  if and only if $\E$ is $2$ times Fr\'echet-differentiable at 
${V}=0$. In that case, $\E'(0)$ and $\E''(0)$ are respectively called the first and second order shape derivative of $E$ at $\Om$, and can also be denoted by $E'(\Om)$ and $E''(\Om)$. Their values at {${V}\in W^{1,\infty}(\R^2;\R^2)$ and $({V_{1},V_{2}})\in{(W^{1,\infty}(\R^2,\R^2))^2}$ are denoted by $E'(\Om)({V})$, $E''(\Om)({V_{1},V_{2}})$, and are called 
first, resp. second, derivative of $E$ at $\Omega$ in the direction ${V}$, resp. ${(V_{1},V_{2})}$}. In this context, it is well known that $\Vol, E_{f},\lambda_{1}$ are twice shape differentiable at any open set $\Om$
{(if $f$ is regular enough)}, and similarly for {the perimeter $P$} if $\Om$ is Lipschitz.
{Note that, roughly speaking, {an open set} is said {to be} Lipschitz if in the neighborhood of every boundary point the set is
below the graph of a certain Lipschitz function, and so its boundary is the graph of this function, see for example \cite{G85Ell}.
Note that  every open convex set is Lipschitz.

The typical example of energy $J$ {that we were able to consider in these previous works \cite{LN,LNP1,LNP2} is} 
\[J(\Om):=F(|\Om|,E_{f}(\Om),\lambda_{1}(\Om))-P(\Om),\]where  $F$ is a smooth function. We showed that $J$ satisfies the {\it concavity property} \eqref{e:suit-conv-prop}.
The difficulty was to study the second order shape derivatives of the involved functional at a set $\Om_{0}$ whose regularity is {only} the one given by the convexity of $\Om_{0}$. Indeed it is easy to verify that the function $\Om\mapsto -P(\Om)$ satisfies the {\it concavity property} \eqref{e:suit-conv-prop} {with $\sigma=1$}, and it {remains} to show that {the term
$\Om\mapsto F(|\Om|,E_{f}(\Om),\lambda_{1}(\Om))$ does not affect the concavity property of $\Om\mapsto-P(\Om)$ (it is a perturbation term)}. For more precise references,
\begin{enumerate}[label=\roman*)]
\item 
in \cite{LN} we only dealt with geometric functionals of the form $F(|\Om|)-P(\Om)$,
\item 
in \cite{LNP1} we showed that if $\Om$ is convex in $\R^2$, then 
\begin{equation}\label{e:E''}
|E_{f}''(\Om)({V,V})|
\leq 
C\|{V}\|^2_{H^{s}\cap L^\infty(\partial\Om)}, \;\;\;\;\;|\lambda_{1}''(\Om)({V,V})|\leq C\|{V}\|^2_{H^{s}\cap L^\infty(\partial\Om)},
\end{equation}
{with $s=1/2$}, which was enough to {conclude that every functional $J$ of the type 
$J(\Om)=F(|\Om|,E_{f}(\Om),\lambda_{1}(\Om))-P(\Om)$ satisfies the concavity property \eqref{e:suit-conv-prop}; 
note that {more generally} any functional
$J(\Om)={{G}(\Om)}-P(\Om)$ with $G''(\Om)$ satisfying
$|{G}''(\Om)({V,V})|\leq C\|{V}\|^2_{H^{s}\cap L^\infty(\partial\Om)}$ with {$s<\sigma$}, will still satisfy \eqref{e:suit-conv-prop}},
\item 
and in \cite{LNP2}, we improved and generalized the previous result, mainly showing that
if $\Om$ is convex in $\R^{{N}}$ then 
\begin{equation}\label{e:E''+}
|E_{f}''(\Om)({V,V})|\leq C\|{V}\|^2_{H^{1/2}(\partial\Om)}, \;\;\;\;\;|\lambda_{1}''(\Om)({V,V})|\leq C\|{V}\|^2_{H^{1/2}(\partial\Om)},
\end{equation}
where the generalization dealt with the dimension of the shapes and also the class of differential operators (uniformly elliptic).
\end{enumerate}
Moreover, 
{computations around smooth shapes led us to believe that} 
$\Om\mapsto -E_{f}(\Om)$ or $\Om\mapsto -\lambda_{1}(\Om)$ also satisfy a {concavity property similar to \eqref{e:suit-conv-prop} with $\sigma=1/2$}, so one could expect that these functionals could assume a similar role as
$\Om\mapsto -P(\Om)$ in the examples above: more precisely, this would mean that one could consider functionals of the type ${G}(\Om)-E_{f}(\Om)$ or ${G}(\Om)-\lambda_{1}(\Om)$, {with ${G}(\Om)$ a perturbation term with a second order shape derivative bounded in $\|\cdot\|_{H^s(\partial\Om)}$ {for some} $s<1/2$,  and still expect (minimizing) optimal shapes to be polygonal\footnote{{In problems \eqref{e:Pb1} to \eqref{e:Pb4},} the functionals we consider have $+E_f(\Om)$ or $+\lambda_1(\Om)$ terms, and we consider the associated maximization problems}}. 
In this paper we will consider only $F(\Omega)=|\Om|$ but the strategy can easily be adapted to other cases.
The main ingredient of proof of this result is an estimate {of the form}
\eqref{e:suit-conv-prop} with $\sigma=1/2$.
Its proof represents many additional technical difficulties, in particular estimates of the kind \eqref{e:E''+} are not sufficient. The main difficulty {is} related to the fact that $\Om$ is only assumed to be convex, so its regularity is a priori weak (and for good reasons, as we expect optimal shapes to be polygonal).
In this case, just computing $E_{f}''(\Om)$ or $\lambda_{1}''(\Om)$ is already challenging (see for example the new results in \cite{L2020}), and the usual way to write these shape derivatives involves integrations by part which are allowed only with more regularity. 

For this kind of problem {\eqref{e:Pb1} to \eqref{e:Pb4},}
only a weak result exists in \cite[Proposition 3.61]{H20Sha}, where 
it is shown (for problems \eqref{e:Pb2} and \eqref{e:Pb4}, but the result and the proof also applies to \eqref{e:Pb1} and \eqref{e:Pb3} if $f\in H^2_{loc}$)}
that the free boundary $\partial{\Om}_0\cap D_{1}\cap D_{2}$ cannot have a $C^2_{+}$ part in its boundary ($C^2_{+}$ means of class $C^2$ with positive (Gauss) curvature); however this result applies in any dimension.
{Therefore} this paper completes the result in \cite[Proposition 3.61]{H20Sha} and {we also apply our results } to new examples.

{\begin{remark}
Let us mention a similar problem in the literature where the maximization of $\lambda_{1}$ is dealt with: in \cite{BF16Bla} the authors are led to study
\begin{equation}\label{eq:bucfra}\max\Big\{\lambda_{1}(\Om)|\Om|, \;\;\Om\in\mathcal{O}\Big\}
\end{equation}
where $\mathcal{O}$ denotes the class of convex planar axisymmetric octagons having four vertices lying on the axes at the same distance, say 1. They show that the square (having vertices at $(\pm1,0)$, $(0,\pm1)$ is the unique solution to this problem.

This problem can be seen as a reverse Faber-Krahn inequality (in a specific class of domains), similarly to \eqref{e:Pb2}-\eqref{e:Pb4}; also the fact that an optimal shape would rather have four sides than eight is a similar behavior to our main result Theorem \ref{th:main} below. The fact that only a specific class of axisymmetric octagons is considered is mainly motivated by the fact that the authors were interested in showing a Mahler-type inequality for $\lambda_{1}$, and that solving \eqref{eq:bucfra} was sufficient for this purpose, see \cite[Proposition 10]{BF16Bla}. Nevertheless, as the authors mention, it would be interesting to consider more general reverse Faber-Krahn inequalities: for example, it is conjectured that the square is the unique (up to affine transformations) solution to
$$\sup\Big\{\inf_{T\in GL_{2}}\big[\lambda_{1}(T(\Om))|T(\Om)|\big],\; \Om\textrm{ nonempty open convex set in  such that }-\Om=\Om\Big\},$$
where the infinimum in $T$ helps bringing an affine transformation invariance to the functional, similarly to Ball's-reverse isoperimetric inequality, see \cite{B91Vol}. Even though it definitely requires some extra work, it is likely that the strategy from our paper could help proving that a solution to the previous problem must be a polygon.
\end{remark}}

\paragraph{Main results.}

Our main result states that the solutions to our problems have a polygonal free boundary. More precisely {we show the following result.}

\begin{theorem}\label{th:main}
 Let $D_{1}\subset D_{2}$ two open convex sets {in $\mathbb R^2$} \red{such that $\overline D_{1}\subset D_{2}$}, $\mu,{m_{0}}\in\R_{+}^*$. 
 If ${\Om}_0$ is a solution of 
 \eqref{e:Pb1}, \eqref{e:Pb2}, \eqref{e:Pb3} or  \eqref{e:Pb4},
where in the case of \eqref{e:Pb1} and \eqref{e:Pb3} we assume 
{\begin{equation}\label{eq:hypf}
f\in H^2_{loc}(\R^2), \;f\gneqq0\textrm{ or }-f\gneqq0,\textrm{ and }
|f|^\beta\textrm{ is concave for some  }\beta\geq1.
\end{equation}}
 Then
 $$
 \textrm{every connected component of }
{\partial{\Om}_0\cap (D_{2}\setminus{\overline{D_{1}}})}\textrm{ is polygonal}.
 $$
{Here, by ``polygonal`` we mean a continuous curve built of a finite number of segments.} 
\end{theorem}

\red{Since the contact set $\partial\Om_{0}\cap(\partial D_{2}\cup\partial D_{1})$ may not be empty, one cannot state without further assumption that the whole set $\Om_{0}$ is a polygon. Nevertheless, we can obtain the following nice result:\footnote{We thank the anonymous referee for suggesting such result.}
\begin{corollary}\label{cor:polygon}
With the same notations and assumptions as in Theorem \ref{th:main}, if furthermore $D_{1}$ and $D_{2}$ are convex polygons, then $\Om_{0}$ is a polygon.
\end{corollary}

\noindent{\bf Proof of Corollary \ref{cor:polygon}:}
Let us denote $\partial\Om_{0}^{in}=\partial\Om_{0}\cap (D_{2}\setminus{\overline{D_{1}}})$ the free boundary, and $\partial\Om_{0}^{out}=\partial\Om_{0}\cap(\partial D_{2}\cup \partial D_{1})$ the contact boundary. We first remark that if $P,Q\in \partial \Om_{0}^{out}$ belong to the same segment either in $\partial D_{1}$ or $\partial D_{2}$, then by convexity, we are in one of the following situations:
\begin{itemize}
\item the whole segment $[P,Q]$ is included in $\partial \Om_{0}^{out}$, or
\item $P$ and $Q$ are corners of $D_{1}$.
\end{itemize} 
It implies that the number of connected components of $\partial \Om_{0}^{out}$ is lower than the sum of the numbers of sides of $\partial D_{1}$ and $\partial D_{2}$.

As a consequence, $\partial\Om_{0}^{in}$ being the complement of $\partial\Om_{0}^{out}$ in $\partial\Om_{0}$, it must also have a finite number of connected components (to see this, one can use polar coordinates centered at a point $O$ lying inside of $D_{1}$ (if $D_{1}\neq \emptyset$) to that $\partial\Om_{0}=\{(\rho(\theta),\theta), \theta\in\T\}$ for some continuous function $\rho:\T\to(0,\infty)$, and realize that $\partial\Om_{0}^{in}$ and $\partial\Om_{0}^{out}$ are parametrized by two complement sets in $\T=\R/2\pi\Z$; if $D_{1}$ is empty, a similar argument works with $O$ lying inside of $\Om_{0}$).

Applying Theorem \ref{th:main}, one knows that each connected component of $\partial\Om_{0}^{in}$ is made of a finite number of segments. So in conclusion, the whole boundary $\partial\Om_{0}$ is made of a finite number of segments, wich concludes the proof.
\qed
}

{\begin{remark}
\red{Assumption \eqref{eq:hypf} allows the case $f=1$ which leads $E_{f}$ to be proportional to the classical torsional energy. This case has been extensively studied from the point of view of geometric inequalities.
It is not clear however whether assumption \eqref{eq:hypf} can be replaced by the much more general assumption $f\in L^2_{loc}(\R^2)$:
\begin{itemize}
\item the regularity assumption on $f$ is made so that $E_{f}$ is twice shape differentiable, and also to get enough regularity for $U$ solution to \eqref{eq:state}  in $\Omega_0$, which is used to apply \cite[Theorem 4.1]{K1985}. As suggested by one the referees, in order to weaken this hypothesis, we tried to apply Theorem \ref{th:main} to a sequence of regular $f_n$ 
converging to a less regular $f$, but it seems we would then need an estimate of the number of segments of each connected component of
the free boundary. But it is not clear how such an estimate would follow from our proof,
\item the sign and concavity assumption for $f$ are made also to use  \cite[Theorem 4.1]{K1985}, which shows that the $\eps$ level set $\{U>\eps\}$ of $U$ are convex. We actually only need this to be true for small values of $\eps$, but it is not clear whether this can be obtained with a weaker assumption.
\end{itemize} 
Notice that the same properties are true for $U_{1}$ even though in this case we do not need any extra assumption.}
\end{remark}

{\begin{remark}
Theorem \ref{th:main} is also valid if one only considers local solutions to  \eqref{e:Pb1}, \eqref{e:Pb2}, \eqref{e:Pb3} or  \eqref{e:Pb4}. Note that a weak version of this fact for \eqref{e:Pb4} when $\S_{ad}=\{\Om\subset\R^2\}$ (that is to say when $D_{1}=\emptyset$ and $D_{2}=\R^2$) has been used in \cite{FL21Bla}: more precisely, it was used that there is no planar $C^{1,1}$-domain that locally (say for the Hausdorff distance) maximizes $\lambda_{1}$ under area and convexity constraint. Note that we actually prove here a stronger statement, namely that any convex shape that is not a polygon cannot be a local maximizer of $\lambda_{1}$ under area and convexity constraint. As far as we know, existence of a local maximum for $\lambda_1$ 
under area and convexity constraints is an open problem in general (see \cite[Lemma 3.5]{FL21Bla} for a solution to the same question for the perimeter instead of $\lambda_{1}$).
\end{remark}}

\begin{remark}\label{r:Om0,exist}
It is classical that problems
\eqref{e:Pb1}, \eqref{e:Pb2}, \eqref{e:Pb3} and 
\eqref{e:Pb4}
always have a solution if $\S_{ad}$ is {nonempty and} compact for the Hausdorff convergence, see \cite{HP2018}. 
This is valid: 
\begin{itemize}
\item 
  for problem \eqref{e:Pb1} and \eqref{e:Pb2} if $D_{2}$ is bounded and $D_{1}$ is nonempty (if $D_{1}$ is empty, for the maximization of  \eqref{e:Pb2} one can consider the empty set to be a solution with infinite energy)
 \item 
 for problem  \eqref{e:Pb3} and \eqref{e:Pb4} if $D_{2}$ is bounded {and $m_{0}{\in (|D_{1}|,|D_{2}|)}$}: {indeed in that case, {even if $D_{1}=\emptyset$, }the volume constraint prevent minimizing sequences to collapse.}
\end{itemize}
\end{remark}

 \begin{remark}\label{r:questions}
 In order to better understand the qualitative properties of solutions  ${\Om}_0$ to 
 \eqref{e:Pb1} { to }\eqref{e:Pb4}, several questions remains. 
 \begin{itemize}
 \item 
 What does the contact sets $\partial{\Om}_0\cap\partial D_{1}$ or $\partial{\Om}_0\cap\partial D_{2}$ look like? Can one prove that for some values of the parameters, these sets are reduced to a finite number of points? (in which case the {whole} optimal shape ${\Om}_0$ is {an actual} polygon).
 \item 
For the penalized problems \eqref{e:Pb1} and \eqref{e:Pb2}, {can one prove} that for $\mu$ large enough the optimal shape is the same as in the limit $\mu\to+\infty$, which means ${\Om}_0=D_{2}$?
{Similarly} can one prove that for $\mu$ small enough, the optimal shapes for \eqref{e:Pb1}, \eqref{e:Pb2} are the same as when $\mu=0$, which means ${\Om}_0=D_{1}$?
 \item 
 Can we compute the number of sides in the polygonal curve  {$\partial{\Om}_0\cap (D_{2}\setminus{\overline{D_{1}}})$}? {How does it evolve with respect to the parameters $\mu, m_{0}$?}
 \end{itemize}
 Also, some numerical computations would help understanding these peculiar optimal shapes, as an analytical computation of optimal shapes seems out of reach (see \cite{BH12Opt} for some computation of optimal shapes for the {maximization of $\Om\mapsto P(\Om)-\mu|\Om|$}).
{These questions will be the topic of future investigations}.
 \end{remark}

{\paragraph{Strategy of proof and plan of the paper.}

In order to prove Theorem \ref{th:main}, we need some preliminary results that are interesting in themselves.
\begin{itemize}
\item 
In Section \ref{sect:derivatives}, we provide new formulas (see Theorems \ref{th:e_f''(0)=} and \ref{th:l_1''(0)=}) for the second order shape derivatives of $E_{f}$ and $\lambda_{1}$. Since we have to consider nonsmooth shapes, one cannot restrict ourselves to normal perturbations {even though they} usually lead to simpler computations. Moreover, integrations by parts need to be handled very carefully as the state functions $U$ and $U_{1}$ are not very smooth {(see Proposition \ref{p:reg(U)})}. Nevertheless, we are able to make appear several terms, one of which, {while it is} very difficult to estimate (because it involves second order derivatives of $U$ or $U_1$ and the curvature of the boundary), {happens to be} non-negative when the shape is convex. {This allows us} to drop it in the use of optimality conditions when proving Theorem \ref{th:main}. 
\item 
In Section \ref{sect:perturbation}, we construct new deformations of convex planar sets: 
{we use previous ideas from \cite[Theorem 4.1]{LP01New} and \cite[Theorem 2.1]{LN} and go deeper in the analysis to show that if $\Om$ is not a polygon, then one can build a perturbation of $\Om$ preserving the convexity 
{(see Definition \ref{def:convex})},
 and which is close to a ``hat'' deformation {(see Definition \ref{d:hat})}, see Proposition \ref{p:varphi}}. These {perturbations are crucial in controlling  the sign of $E_{f}''(\Om)$ and $\lambda_{1}''(\Om)$}, {see Theorem \ref{th:e_f''(0)->convex}}.
 
 \begin{definition}\label{def:convex}
Let $\Om$ be a convex set in $\R^2$ and $V\in W^{1,\infty}(\Om;\R^2)$. We say that $V$ preserves the convexity of $\Om$ if there exists $t_{0}>0$ such that $(\Id+tV)(\Om)$ is convex for every $t\in[-t_{0},t_{0}]$.
 \end{definition}

All along this paper we use hat functions and {hat perturbations}, which are introduced by this definition.
\begin{definition}\label{d:hat}
\begin{enumerate}
\item 
Let $I\subset \R$ be an interval and $x_1,x_2,x_3\in I$ with $x_1<x_2<x_3$. The function $\varphi\in W^{1,\infty}(I)$ is said to be the {\em hat function associated with the nodes $x_1,x_2,x_3$ }if
$$\varphi\equiv 0 \;in\; I\setminus(x_1,x_3),$$
 and
 \[
\varphi(x)=\left\{
 \begin{array}{l}
\displaystyle{ \frac{x-x_1}{x_2-x_1},\;\forall x\in [x_1,x_2],}
 \vspace*{2mm}
 \\
\displaystyle{\frac{x_3-x}{x_3-x_2},\;\forall x\in [x_2,x_3].}
 \end{array}
 \right.
 \]
{\item
Let $\Om$ an open convex set in $\R^2$, $\Gamma \subset \partial\Om$ and $u:I\to\R$ (with $I$ an interval) such that $\Gamma=\{(x,u(x)), x\in I\}$. Let $\varphi\in W^{1,\infty}(I)$ a hat function associated to three nodes $x_{1},x_{2},x_{3}\in I$. The function $v\in W^{1,\infty}(\partial\Om)$ is called the ``hat function on $\partial\Om$'' associated to $x_{1},x_{2},x_{3}$ (or to the points $p_{i}=(x_{i},u(x_{i}))$, $i\in\{1,2,3\}$) if
\begin{equation}\label{eq:hatdef}
\forall x\in I, \;v(x,u(x))=\varphi(x)
\;\;
\textrm{ and }
\;\;
\forall p\in \partial\Om\setminus{\Gamma},\; v(p)=0.
\end{equation}
In that case, given $z\in \R^2$ a fixed vector, we say that $V=vz$ is the hat deformation of $\partial\Om$ associated to $x_{1},x_{2},x_{3}$ in direction $z$.}
\end{enumerate}
\end{definition}
\item 
In Section \ref{ss:convexity-e''} we prove Theorem \ref{th:main} in the case of problem \eqref{e:Pb1} {and \eqref{e:Pb2}}. {We also prove } the following results, whose interest might go beyond its application to Theorem \ref{th:main}.

{The following theorem  involves quite technical assumptions, 
which roughly speaking are: (i) the perturbation is nearly normal and the boundary is flat enough;
(ii) some first order optimality conditions are satisfied; and
(iii) an assumption on $f$ required to ensure the convexity of level curves of $U$ and estimate the sign of a difficult term of $E_f''(\Om_0)(V,V)$.
}

{
\begin{theorem}\label{th:e_f''(0)->convex}
Let $\Om_{0}$ be a bounded convex open set in $\R^2$ such that with respect to a coordinate system with origin at 
$O\in\partial\Om_0$ we have $O=(0,0)$ and $\Omega_0\subset\{(x,y),\; y>0\}$.
Let also $u_0\in W^{1,\infty}(0,\sigma_0)$ with $\Gamma_0:=\{(x,u_0(x)),\, x\in(0,\sigma_0)\}\subset\partial\Om_0$,
$u(0)=0$ and $u_0'(0^+)=0$. Furthermore, we assume {\eqref{eq:hypf}} and:
\begin{enumerate}[label=(\roman*)]
\item {there exist three points $0<x_1<x_2<x_3<\sigma_0$ with:\\
\hspace*{10mm}
(i.1) $u_0''$ having Dirac masses at each of $x_i$, or \\
\hspace*{10mm}
(ii.2) each of $x_i$ is an accumulation point of ${\rm supp}(u_0'')$,
}
\item
there exists $\mu\in\mathbb R$ such that the following first order optimality conditions hold:
\begin{eqnarray*}
\int_{x_i}^{x_{i+1}}
\left[\frac{1}{2}(\partial_{\nu_0}{U_0})^2-\mu\right]\ell dx
&=&0,
\;\;
\mbox{$\forall \ell$ {affine} in $[x_i,x_{i+1}]$, $i=1,2$},
\end{eqnarray*}
where 
$U_0$ is the solution of \eqref{e:Ef} in $\Om_{0}$, and $(\partial_{\nu_0}{U_0})(x):=(\partial_{\nu_0}{U_0})(x,u_0(x))$, 
\end{enumerate}

Then if $x_3$ is small enough we have
\begin{eqnarray}
E_{f}''(\Om_{0})(V,V)
&\geq&
\int_{\Om_{0}}
|\nabla U_{0}'|^2 
+
\int_{\partial\Om_{0}}
\big(
f(\partial_{\nu_{0}} U_{0})(V\cdot {\nu_{0}})^2
\nonumber\\
&&
\hspace*{34mm}
+|\nabla U_{0}|^2 (V\cdot\tau_{0})(\partial_{s_0} V\cdot \nu_{0})
\big)
>
0,
\label{eq:convEf}
\end{eqnarray}
where $V\in W^{1,\infty}(\R^2;\R^2)$ is the extension given by Lemma \ref{l:Hv} of
the hat deformation associated to nodes $x_{1},x_{2},x_{3}$ in direction $(0,1)$ (see Definition \ref{d:hat}), $\nu_{0}$ is the unit exterior normal vector, $\tau_{0}=\nu_{0}^\perp$, $U_{0}'\in H^1(\Om_{0})$ is solution to \eqref{e:U'}
and $\partial_{s_0}$ denotes the derivative with respect to the arclength $s_0$ of $\partial\Om_0$.
\end{theorem}
}
{
\begin{theorem}\label{th:l_1''(0)->convex}
{Let us assume the same setting as} in Theorem \ref{th:e_f''(0)->convex}{, and where $U_{0}$ is replaced by $U_{1}$ (solution to \eqref{e:l1}) in condition} (ii).
Then 
\begin{eqnarray}\nonumber
\lambda_1''(\Om_{0})(V,V)\geq \int_{\Om_{0}}
|\nabla U_{1}'|^2\!\!\!\!
&+&
\!\!\!\!\!
\int_{\partial\Om_{0}}
\Big(
{\frac{1}{2}\lambda_{1}(\Om_{0}) (U_{1}{\partial_{\nu_{0}} U_{1})}(V\cdot\nu_{0})^2}\\
&&\hspace{8mm}
+
(\partial_{\nu_{0}}U_{1})^2 (V\cdot\tau_{0})(\partial_{s} V\cdot \nu_{0})
\Big)
>0.
\label{eq:convl1}
\end{eqnarray}
{where $U_{1}'\in H^1(\Om_{0})$ is solution to \eqref{e:l',U'}.}
\end{theorem}
}

These results can be seen as a weak convexity result for $E_{f}$ {and $\lambda_1$}. 
{Note that the first part of inequalities \eqref{eq:convEf} and \eqref{eq:convl1} are valid in much more generality, see Theorems \ref{th:e_f''(0)=} and \ref{th:l_1''(0)=}.}
{However,} it remains an open problem whether {$E_{f}''(\Om)(V,V)>0$ and $\lambda_{1}''(\Om)(V,V)>0$ are} valid under weaker assumption, for example  assuming only that $V$ has small support (this is the case if $\Om_{0}$ is smooth, 
{see \cite[Proposition 3.61]{H20Sha}}). 

{In order to prove Theorem \ref{th:main} in the case of problem \eqref{e:Pb1} (and similarly for \eqref{e:Pb2}), we deduce from Theorem \ref{th:e_f''(0)->convex}} that if $\Om_{0}$ is a solution to \eqref{e:Pb1} {then every connected component of $\partial\Om_{0}\cap D_{2}\setminus\overline{D_{1}}$ is polygonal: {assuming by contradiction that it is not the case, } 
{we use} the results from Section \ref{sect:perturbation} {and} the second order optimality condition to show $E_{f}''(\Om_{0})(V,V)\leq 0$ for $V$ a certain hat function, while at the same time Theorem \ref{th:e_f''(0)->convex} {applies, leading to a contradiction}.}

\item
Finally, in Section \ref{sect:other} we show how to adapt the previous strategy to problems  \eqref{e:Pb3} and \eqref{e:Pb4}. {The main difficulty is to handle the volume constraint for problems \eqref{e:Pb3} and \eqref{e:Pb4}, and show that the key lemmas \ref{l:...phi'phi>=0}, \ref
{l:int|DU|phi^2<},  \ref{l:varphi} hold for certain volume preserving perturbations. This works with some adaptations, one of which requires a classical regularity result for free boundary problems.}
\end{itemize}

We conclude this section with the following lemma, which gives an extension result used often all along this paper.
This result is classical but we give its proof for the sake of completeness. {Given $\Om_{0}$ an bounded, open and nonempty convex set, we can choose an origin $O\in\Om_{0}$ and consider $r_{0}:\T\to(0,+\infty)$ the radial function of $\Om_{0}$ (where $\T:=\R/2\pi\Z$), that is to say
$$\forall \theta\in\T, \;\;r_{0}(\theta)=\sup\big\{r>0, re^{i\theta}\in \Om_{0}\big\}.$$
This allows us to conveniently define $W^{1,\infty}(\partial\Om_{0})$ in the following way:
\begin{equation}\label{e:|W1inf(pO)|}
W^{1,\infty}(\partial\Om_0)
=
\Big\{
v:\partial\Om_0\mapsto\mathbb R,
\;
h_v(\theta):=v(r_0(\theta)\cos\theta,r_0(\theta)\sin\theta)\in W^{1,\infty}(\T)\Big\}
\end{equation}
endowed with the norm
$$\|v\|_{W^{1,\infty}(\partial\Om_0)}:=\|h_v\|_{W^{1,\infty}(\T)}.$$

}

\begin{lemma}\label{l:Hv}
({\bf An extension operator})
Let $\Om_0\subset\R^2$ be bounded and convex.
Then there exists a linear continuous operator (extension)
$H:v\in W^{1,\infty}(\partial\Om_0)\mapsto \bar{v}\in W^{1,\infty}({\mathbb R^2})$ such that
$\bar{v}=v$ on $\partial\Om_0$, 
{ and $\partial_r\nabla \overline{v}\in L^\infty(\R^2;\R^2)$, where $\partial_r$ is the radial derivative.}
\end{lemma}
{Note that the particular property of this extension (namely $\partial_r\nabla \overline{v}\in L^\infty(\R^2;\R^2)$) will be used for computing
one term in $E_{f}''(\Om_{0})$, namely the limit of $K_4(\eps)$ in the proof of Theorem \ref{th:e_f''(0)=}}.
\begin{proof}
Let $r_0\in W^{1,\infty}(\T)$ be the {radial} function of $\Omega_0$.
For $v\in W^{1,\infty}(\partial\Om_0)$ we consider its $W^{1,\infty}({\mathbb R^2})$ extension $\overline{v}$ defined by
\begin{equation}
\bar{v}(x,y)=(Hv)(x,y)
:=
\eta\left(\frac{r}{r_0(\theta)}\right)h_v(\theta),\quad\forall 
(x,y)\in \R^2,\;\; x=r\cos\theta,\; y=r\sin\theta,
\label{e:Hv}
\end{equation}
where $\eta\in {\cal D}(\R)$ with $0\leq\eta\leq 1$, $\eta\equiv 1$ 
in a neighborhood of $1$ and $\eta \equiv 0$ {elsewhere}. 
Note that if $(x,y)=(r\cos\theta,r\sin\theta)$, {from \eqref{e:Hv} and  using
{\[
 \partial_{x}r=\cos \theta,\;\; \partial_{x}\theta=-\sin\theta/r,\quad
 \partial_{y}r=\sin\theta,\;\; \partial_{y}\theta=\cos\theta/r,
\]}}
we get
\begin{eqnarray}
\partial_x \overline{v}(x,y)
&=&
\eta'\left(\frac{r}{r_0(\theta)}\right)\frac{1}{r_0(\theta)}h_v(\theta)\cos\theta
\nonumber\\
&&+
\left(
\eta'\left(\frac{r}{r_0(\theta)}\right)\frac{r_0'(\theta)}{r_0(\theta)^2}h_v(\theta)
-
\frac{1}{r}\eta\left(\frac{r}{r_0(\theta)}\right)
h'_v(\theta)
\right)
\sin\theta
\in L^\infty(\R^2),
\label{e:DxHv}
\end{eqnarray}
and
\begin{eqnarray}
\partial_y \overline{v}(x,y)
&=&
\eta'\left(\frac{r}{r_0(\theta)}\right)\frac{1}{r_0(\theta)}h_v(\theta)\sin\theta
\nonumber\\
&-&
\left(
\eta'\left(\frac{r}{r_0(\theta)}\right)\frac{r_0'(\theta)}{r_0(\theta)^2}h_v(\theta)
-
\frac{1}{r}\eta\left(\frac{r}{r_0(\theta)}\right)
h'_v(\theta)
\right)
\cos\theta
\in L^\infty(\R^2).
\label{e:DyHv}
\end{eqnarray}
So $v\in W^{1,\infty}(\partial\Om_0)\mapsto \overline{v}\in {W^{1,\infty}(\R^2)}$ 
is a well-defined continuous extension operator.

Also $\partial_r \nabla v\in W^{1,\infty}(\R^2;\R^2)$ follows easily from {\eqref{e:DxHv}-\eqref{e:DyHv},} and noticing that the only functions
in these formulas depending on $r$ are $\eta$ and $\eta'$.

\end{proof}
}

{\section{New computations of second order shape derivatives}\label{sect:derivatives}}

{In the smooth case, it is usual to express first and second order shape derivatives as boundary integrals, in terms of $V\cdot\nu$  (where $\nu$ denotes the unit exterior normal vector to $\partial\Om$) and $V_{\tau}=V-(V\cdot\nu)\nu$, see for example \cite[Chapter 5]{HP2018}. This is due to the so-called Hadamard structure theorem (see \cite{NovruziPierre,DL}), and also to the fact that the signs of the different terms appear more naturally in {terms of $V\cdot\nu$ and $V_\tau$}. In particular, it is a general fact that restricted to normal deformations (i.e. such that $V_{\tau}=0$), second order shape derivatives are symmetric quadratic forms in terms of $(V\cdot\nu)$.

 In our nonsmooth framework, we face two difficulties. {{The first is to} express these derivatives as boundary integrals, since the expression of these integrals require integration by parts and} high regularity of the state function $U$ and $U_{1}$ (see also \cite{L2020}). Second, {as $\nu$ is not smooth enough, in order to prove Theorem \ref{th:main} we will not use normal deformations. Using some new computation techniques we will find useful expressions for $E_f''(\Om)(V,V)$ and
 $\lambda_1''(\Om)(V,V)$ for $V\in W^{1,\infty}(\mathbb R^2,\mathbb R^2)$.}\\

\subsection{Shape derivatives of the volume}

Let us start with the derivative of the volume functional.
We will use the following notations: if $(a,b)\in\R^2$, then $(a,b)^{\perp}:=(-b,a)$ (counterclockwise rotation with angle $\pi/2$). In particular, $\tau=\nu^\perp$ is the tangent vector to $\partial\Om$. {We also denote {$\nabla^\perp\varphi{=(\partial_1^\perp\varphi,\partial_2^\perp\varphi)}=(-\partial_{2}\varphi,\partial_{1}\varphi)$} for $\varphi\in H^1(\Om)$.}
Moreover, for $v\in W^{1,\infty}(\partial\Om)$, we denote 
{$\partial_s v:=\nabla\overline{v}\cdot\tau$  the derivative with respect to the arclength 
in $\partial\Omega$, where $\overline{v}$ is any 
$W^{1,\infty}(\mathbb R^2)$-extension of $v$, {such as the one in Lemma \ref{l:Hv}}.}

\begin{proposition}\label{prop:vol''}
Let $\Om$ be a Lipschitz set in $\R^2$ and {$V=(V_1,V_2)\in W^{1,\infty}({\R^2};\R^2)$}.
Then 
\begin{eqnarray}
{\Vol''(\Om)(V,V)}
&=&
\int_{\partial\Om}
{(\partial_s V\cdot V^\perp)}.
\label{e:m''(0)}
\end{eqnarray}
If in particular 
 $V=vz$, where $v\in W^{1,\infty}(\R^2)$, 
  and $z\in W^{1,\infty}(\mathbb R^2;\mathbb R^2)$ then
\begin{equation}\label{e:m''(0)bis}
{\Vol''(\Om)(V,V)}
=
\int_{\partial\Om}
 {(\partial_s z\cdot z^\perp)}v^2.
 \end{equation}
\end{proposition}
\begin{remark}
If $\Om$ is smooth enough, one can choose {$V$ {a Lipschitz extension of} $v\nu$}. Then $z^\perp=\tau$ and classically 
$\partial_{s}\nu=\mathcal{H}\tau$, where $\mathcal{H}$ is the curvature.
Therefore \eqref{e:m''(0)bis} {leads to} the usual formula
\[\Vol''(\Om)(V,V)=\int_{\partial\Om}
\mathcal{H}v^2\]
\end{remark}

\begin{proof}
From $m(t):=\Vol((\Id+tV)(\Om))=\int_{\Om}\det(\Id+tDV)dx=\int_{\Om}(1+t(\nabla\cdot V) + t^2\det(DV))dx$  one gets easily
\begin{eqnarray*}
m''(0)
&=&
2\int_{\Om}{\rm det}({D} V) 
= 
2\int_{\Om}\partial_1V_1\partial_2V_2-\partial_1V_2\partial_2V_1
\\
&=& 
2\int_{\Om}\nabla V_2\cdot\nabla^\perp V_1.
\end{eqnarray*}
Note that for $A,B\in H^1(\Om)$ we have 
$\nabla A\cdot\nabla^\perp B
=\nabla\cdot(A\nabla^\perp B)
=-\nabla\cdot(B\nabla^\perp A)\in L^2(\Om)$. These equalities can be proven by taking a regularizing sequence in $H^1(\Om)$ for $A$ and $B$ and passing {to the} limit. Then using the integration by part, see \cite[Th. 2.5 and (2.17)]{GR88FEMNS}
\begin{eqnarray*}
m''(0)
&=&
\int_{\Om}
(\nabla\cdot(V_2\nabla^\perp V_1)
-
\nabla\cdot(V_1\nabla^\perp V_2))
\\
&=&
\lara{V_2(\nabla^\perp V_1\cdot\nu) - V_1(\nabla^\perp V_2\cdot\nu),1}_{H^{-1/2}(\partial\Om)\times H^{1/2}(\partial\Om)}
\\
&=&
\int_{\partial\Om}
(V_2(\nabla^\perp V_1\cdot\nu) - V_1(\nabla^\perp V_2\cdot\nu))
\\
&=&
\int_{\partial\Om}
V_1\partial_s V_2 - V_2\partial_s V_1
\\
&=&
\int_{\partial\Om}
\partial_s V\cdot V^\perp,
\end{eqnarray*}
which proves \eqref{e:m''(0)}. If $V=v z$ then 
\begin{eqnarray*}
m''(0)
&=&
\int_{\partial\Om}
(z_1\partial_sz_2-z_2\partial_sz_1) v^2
=
\int_{\partial\Om}
(\partial_s z\cdot z^\perp) v^2
,
\end{eqnarray*}
which proves \eqref{e:m''(0)bis}.
\end{proof}

\subsection{Shape derivatives of Dirichlet energy}

{
First, we recall the following classical result for the solution $U$ of \eqref{e:Ef} (see \cite{Kadlec} for a proof (see also \cite{LNP1})).
\begin{proposition}\label{p:reg(U)}
For every $\Omega\subset\R^2$ {open bounded convex set, the unique solution $U$ of \eqref{e:Ef} satisfies} $U\in W^{1,\infty}(\Om)\cap H^2(\Omega)$.
\end{proposition}
{This means that boundary second order derivatives of $U$, which appear on the second order shape derivative of $E_f$,
are not well-defined in general even in the case when $\Omega$ is $C^2$ and makes difficult 
the computation of $E_f''(\Om)$ in terms of boundary integrals.}
The following result gives a formula for $E_{f}''(\Om)$, which, due to the limited regularity of $U$, is quite technical to obtain.}
Nevertheless, it allows us to identify several terms in $E_{f}''(\Om)$, for which we will be able to identify signs. 
It happens to be crucial to prove \eqref{eq:convEf} in Theorem \ref{th:e_f''(0)->convex}.

\begin{theorem}\label{th:e_f''(0)=}
Let $\Omega\subset\R^2$ be a bounded open  {convex} set and
for $v=(v_1,v_2)\in W^{1,\infty}(\partial\Om;\R^2)$ set 
$V=(Hv_1,Hv_2)\in W^{1,\infty}(\R^2;\R^2)$ as given by Lemma \ref{l:Hv}. 
Assume \eqref{eq:hypf}, 
and {$U_0\in W^{1,\infty}(\Om_0)\cap H^2(\Om_0)$} is the solution of \eqref{e:Ef} in $\Om_0$.
Then $E_{f}$ is twice shape differentiable at $\Om$ and
\begin{equation}\label{e:e_f''(0),main}
E_{f}''(\Om)(V,V)
\geq \int_{\Om}
|\nabla U'|^2 
+
\int_{\partial\Om}
\left(
f(\partial_{\nu} U)(V\cdot {\nu})^2
+
|\nabla U|^2 (V\cdot\tau)(\partial_{s} V\cdot \nu)
\right),
\end{equation}
where $\nu$ is the exterior unit normal vector on $\partial\Om$, $\tau=\nu^\perp$,
$\partial_{s}(\cdot)=\nabla(\cdot)\cdot\tau$, 
$U'\in H^1(\Om)$ solves (in the variational sense)
\begin{equation}
\left\{
\begin{array}{cccc}
-\Delta U' &=&0\;\; &\mbox{ in }\;\;\Om,\\
U' &=& -\nabla U\cdot V\;\;&\mbox{ on }\;\;\partial\Om.
\end{array}
\right.
\label{e:U'}
\end{equation}
\end{theorem}

The proof of this result is given below, but first we give some remarks to understand better its significance.
{
\begin{remark}\label{rk:K3} This result will be sufficient for our purpose, which is to prove Theorem \ref{th:main}. But as we will detail in an appendix, we are able to get a complete formula for $E_{f}''(\Om)$:
\begin{enumerate}[label=(\roman*)]
\item 
In the case $\Omega$ is {$C^{2,\sigma}$}, for any $\sigma\in(0,1)$, we can actually prove (see Proposition \ref{p:E_f''-complete})
\begin{eqnarray}
\hspace*{-6mm}
E_{f}''(\Om)(V,V)
\!\!\!\!\!&=&\!\!\!\!\!
\int_{\Om}
|\nabla U'|^2 
\nonumber\\
\hspace*{-6mm}
\!\!\!\!\!&&\!\!\!\!\!+
\int_{\partial\Om}
\left(
f(\partial_{\nu} U)(V\cdot {\nu})^2
+
\frac{1}{2}
{\cal H}(\partial_{\nu} U)^2|V|^2
+
|\nabla U|^2 (V\cdot\tau)(\partial_{s} V\cdot \nu)
\right)\label{eq:Ef''}
\end{eqnarray}
which, to the best of our knowledge was not known. Note that when $V$ is normal on $\partial\Om$, we retrieve the usual well-known formula, see for example \cite{HP2018}.
\item
When $\Om$ is not {smooth}, it is not easy to make sense of the term 
$\int_{\partial\Om}{\cal H}(\partial_{\nu} U)^2|V|^2$ {as it involves the product of ${\cal H}$ with $(\partial_\nu U)^2$,
which at a corner for example, are respectively unbounded and zero}.
It is worth noticing that we are actually able to give sense to this term {when $\Omega$ is convex} (see the appendix for more details): more precisely, we prove that the following limit exists
\begin{equation}\label{eq:weakH}
\lim_{\eps\to 0}\int_{\partial\Om_{\eps}} {{\cal H}_\eps} (\partial_{\nu_\eps} U)^2 |V|^2
\end{equation}
where {$\Om_{\eps}=\{|U|>\eps\}$},
$\nu_\eps$ is the exterior unit normal vector to 
$\partial\Om_\eps$ 
{(we restrict ourselves to values of $\eps$ so that $\Om_{\eps}$ is smooth)} and {$\mathcal{H}_\eps$} denotes the curvature of $\partial\Om_{\eps}$ which is defined in the classical sense. Defining $\int_{\partial\Om}{\cal H}(\partial_{\nu} U)^2|V|^2$ as this limit, formula \eqref{eq:Ef''} remains valid.
\item {From {\eqref{eq:weakH}} it is clear that the term $\int_{\partial\Om}{\cal H}(\partial_{\nu} U)^2|V|^2$
 is nonnegative, which explains \eqref{e:e_f''(0),main}}.
\end{enumerate}
\end{remark}}

In order to prove Theorem \ref{th:e_f''(0)=}, we will proceed in several steps. First, in Proposition \ref{p:e_f'(0),e_f''(0)} we recall usual formulas for first and second order shape derivatives, and then we proceed to the proof of Theorem \ref{th:e_f''(0)=}.

\begin{proposition}\label{p:e_f'(0),e_f''(0)}
Let $\Om$ be an open bounded Lipschitz subset of $\R^2$, $f\in H^2_{loc}(\R^2)$ and $V\in W^{1,\infty}(\R^2;\mathbb R^2)$. 
{For $t\in\R$, $|t|$ small, let $U_t$ be the solution of \eqref{e:Ef} in $(I+tV)(\Om)$ and $\hat{U}_t:=U_t\circ(\Id+tV)$. 
Then $t\mapsto E_f((Id+tV)(\Om))$ and $t\mapsto \hat{U}_t\in H^1_{0}(\Om)$ are $2$-times differentiable at $0$ and}
\begin{eqnarray}
{E_{f}'(\Om)(V)}&=&
-\frac{1}{2}
\int_{\partial\Om}{|\nabla U|^2}(V\cdot\nu),
\label{e:e_f'}
\end{eqnarray}
and if furthermore $U\in H^2(\Om)$ we have 
\begin{eqnarray}
{E_{f}''(\Om)(V,V)}
&=&
\frac{1}{2}
\int_{\partial\Om}
f({\partial_{\nu}U})(V\cdot \nu)^2
-
\frac{1}{2}
\int_{\Omega}fU''dx.
\label{e:e_f''}
\end{eqnarray}
Here 
{$U'\in H^1(\Om)$ solves \eqref{e:U'} and satisfies
\begin{eqnarray}
U'=\hat{U}'-\nabla U\cdot V\;\; \mbox{in}\;\; \Om, \label{e:U'=}
\end{eqnarray}
and $U''\in L^2(\Om)$  is given by
\begin{eqnarray}
-\Delta U''&=&0\;\; \mbox{in}\;\; {{\cal D}'(\Om)},\quad  
U''=\hat{U}''-2(\nabla U'\cdot V) - (V\cdot D^2U\cdot V)\;\; \mbox{in}\;\; \Om, \label{e:U''=}
\end{eqnarray}
where $\hat{U}'\in H^1_{0}(\Om)$, resp. $\hat{U}''\in H^1_{0}(\Om)$, are the  first, resp. second, order derivative 
of  $t\mapsto \hat{U}_t$ at $t=0$.}
\end{proposition}
{\bf Proof}.
{
See for example \cite{HP2018} for the proof of shape differentiability for $E_{f}((Id+t)(\Om))$ and $\hat{U}_t$. The proof of \eqref{e:e_f'} follows from \cite[Proposition 3]{LNP1}, 
\eqref{e:e_f''} and \eqref{e:U''=} follow from \cite[Lemma 5]{LNP1} and
\eqref{e:U'=} follows from  \cite[Lemma 3]{LNP1}.}
\hfill$\Box$
\\

\noindent\textbf{Proof of Theorem \ref{th:e_f''(0)=}.}
{We will consider the case $f\gneqq0$ and $f^\beta$ concave. 
The other case is proved similarly by considering $-f$ instead of $f$.}
Note that from $f\in H^2_{loc}(\mathbb R^2)$ we have 
$U\in H^4_{loc}(\Om)\cap C^{2,\sigma}(\Omega)$, for every $\sigma\in(0,1)$. 
Furthermore, $U', U''\in C^\infty(\Om)$.
From \eqref{e:e_f''} we have
\begin{eqnarray}
{E_{f}''(\Om)(V,V)}
&=&
\frac{1}{2}
\left(
\int_{\partial\Om}
f(\partial_{\nu}U)(V\cdot \nu)^2 ds
+J
\right),\quad
J=
-
\int_{\Om}fU''dx.
\label{e:e_D''=..+J}
\end{eqnarray}
For the estimation of $J$, we will use the sets $\Om_\eps=[{U}>\eps\}$.
{As $f\gneqq0$, $U>0$ in $\Omega$ and as $f^\beta$ is assumed to be concave,  \cite[Theorem 4.1]{K1985} shows that $U^{{\frac{\beta}{1+2\beta}}}$
is concave.
Combined with Sard's Theorem and implicit function Theorem, we conclude that $\Om_\eps$ are $C^{2,\sigma}$-convex domains for almost all $\eps$. Moreover, we easily notice that} $\lim_{\eps\to 0} \mathbbmss{1}_{\Omega_\eps}=1$  in $L^1(\Om)$
so that
\[
J=
-\lim_{\eps\to 0} \int_{\Om_\eps} U''f
=
\lim_{\eps\to 0} \int_{\Om_\eps} U''\Delta U.
\]
Using \eqref{e:U''=} we get 
\begin{eqnarray}
J
&=&
\lim_{\eps\to 0} 
\int_{\Om_\eps}
U''\Delta U
\\
&=&
\lim_{\eps\to 0} 
\int_{\partial\Om_\eps}
(U''(\nabla U\cdot\nu_\eps)-U(\nabla U''\cdot\nu_\eps))
+
\int_{\Om_\eps}
U\Delta U''
\nonumber\\
&=&
\lim_{\eps\to 0} 
\int_{\partial\Om_\eps}
U''(\nabla U\cdot\nu_\eps)
\label{e:like-lambda}
\\
&=&
\lim_{\eps\to 0} 
\int_{\partial\Omega_\eps}
(\hat{U}''-2(\nabla U'\cdot V)-(V\cdot D^2U\cdot V))(\nabla U\cdot \nu_\eps)
\nonumber\\
&=&
\lim_{\eps\to 0} 
(J_1(\eps)+J_2(\eps)+J_3(\eps)),
\label{e:J}
\end{eqnarray}
where $\nu_\eps$ is the exterior normal unit vector on $\partial\Om_\eps$.
{For $J_{1}$, we proceed as in \cite[Proof of Proposition 6]{LNP1}: as $\hat{U}''\in H^1_{0}(\Om)$ one can write
\begin{equation}
J_1(\eps) 
 =
 \int_{\Om_\eps}
 \nabla\cdot(\hat{U}''\nabla U)
 \mathop{\longrightarrow}_{\eps\to 0}
 \int_{\Om}
 \nabla\cdot(\hat{U}''\nabla U)
=
 \int_{\partial\Om}
 \hat{U}''(\nabla U\cdot \nu)
 =
0.
\label{e:J1}
\end{equation}}
For $J_2(\eps)$ we decompose $\nabla U'$ in normal and tangential parts as follows
\begin{eqnarray}
J_2(\eps)
&=&
-2
\int_{\partial\Om_\eps}
(\nabla U'\cdot\nu_\eps)(V\cdot \nu_\eps)(\nabla U\cdot\nu_\eps) 
-
2\int_{\partial\Om_\eps}
(\nabla U'\cdot\tau_\eps)(V\cdot\tau_\eps)(\nabla U\cdot\nu_\eps) 
\nonumber\\
&=&
J_{21}(\eps)+J_{22}(\eps).
\label{e:J21+J22}
\end{eqnarray}
{For $J_{21}(\eps)$ we use the divergence theorem, $\Delta U'=0$ and $\hat{U}=0$ on $\partial\Om$ as follows
\begin{eqnarray}
J_{21}(\eps)
&=&
-2
\int_{\partial\Om_\eps}
(\nabla U'\cdot\nu_\eps)(\nabla U\cdot V) 
=
-2
\int_{\Om_\eps}
\nabla\cdot(\nabla U'(\nabla U\cdot V)) 
=
-2
\int_{\Om_\eps}
\nabla U'\cdot\nabla(\nabla U\cdot V) 
\nonumber\\
&\displaystyle{\mathop{\longrightarrow}_{\eps\to0}}&
-2
\int_{\Om}
\nabla U'\cdot\nabla(\nabla U\cdot V) 
=
-2
\int_{\Om}
\nabla U'\cdot\nabla(\hat{U}'-U')
\nonumber\\
&&\qquad=
-2
\int_{\Om}
\nabla U'\cdot\nabla\hat{U}'
+2
\int_{\Om}
|\nabla U'|^2
=2
\int_{\Om}
|\nabla U'|^2,
\label{e:J21}
\end{eqnarray}
where $\hat{U}'$ is the derivative of $\hat{U}_t$ at $t=0$ and 
$U'=\hat{U}'-\nabla U\cdot V$ in $\Om$.

{For the upcoming computations, we notice that (see notations at the beginning of Section \ref{sect:derivatives})}
\[
\nu^\perp=\tau,
\quad
\tau^\perp=-\nu,
\quad
a\cdot b =a^\perp\cdot b^\perp,
\quad
\partial_{s} g=\nabla g\cdot\tau.
\]
\black{Equalities above are also valid on $\partial\Omega_\eps$ with $\nu_\eps$ and $\tau_\eps$ the exterior normal and tangent vector on $\partial\Omega_\eps$, $\tau_\eps=\nu_\eps^\perp$. 
The following results hold on $\partial\Om_\eps$}
\begin{equation}
(i)\;
\left\{
\begin{array}{rcl}
\partial_{s_\eps}\tau_\eps &=& -{\cal H}_\eps\nu_\eps,
\vspace*{2mm}\\
 \partial_{s_\eps} \nu_\eps &=& {\cal H}_\eps\tau_\eps, 
\end{array}
\right.
\quad
(ii)\;
\left\{
\begin{array}{rclccrlc}
 \partial_{\tau_\eps\tau_\eps} U :=\tau_\eps\cdot D^2U\cdot\tau_\eps
 &=& {\cal H}_\eps(\nabla U\cdot\nu_\eps),
\vspace*{2mm}\\
 \partial_{\nu_\eps\nu_\eps} U :=\nu_\eps\cdot D^2U\cdot\nu_\eps
 &=& -f - {\cal H}_\eps(\nabla U\cdot\nu_\eps),
\end{array}
\right.
\label{e:DnnU}
\end{equation}
because $-\Delta U=f$ in $\Omega$, $U\in C^{2,\sigma}(\Omega)$ for all $\sigma\in(0,1)$, 
$U=\eps$ on $\partial\Om_\eps$  
{and $\partial\Omega_\eps$ { is $C^{2,\sigma}$}.
Formulas $(i)$ are classical and those in $(ii)$ follow easily by considering locally $\partial\Omega$ as the graph of a 
$C^{2,\sigma}$ function {$u:(0,\sigma)\to\R$} with $u'(0)=0$ and $U(x,u(x))=\eps$, and then differentiating this equation twice at $x=0$.}

We decompose $J_{22}(\eps)$ in two terms, one involving 
$\nabla\hat{U}'$ and the other $\nabla U$. 
For one we use the divergence theorem and for the other we integrate by parts
on the boundary as follows
\begin{eqnarray}
J_{22}(\eps)
&=&
-2
\int_{\partial\Om_\eps}
(\nabla (\hat{U}'-\nabla U\cdot V)\cdot\tau_\eps)(V\cdot\tau_\eps)(\nabla U\cdot{\nu_\eps}) 
\nonumber\\
&=&
-2
\int_{\partial\Om_\eps}
(\nabla\hat{U}'\cdot\tau_\eps)(V\cdot\tau_\eps)(\nabla U\cdot\nu_\eps) 
+2
\int_{\partial\Om_\eps}
(\nabla(\nabla U\cdot V)\cdot\tau_\eps)(\nabla U\cdot\nu_\eps)(V\cdot\nu_\eps)\frac{V\cdot\tau_\eps}{V\cdot\nu_\eps} 
\nonumber\\
&=&
2
\int_{\partial\Om_\eps}
(\nabla^\perp\hat{U}'\cdot\nu_\eps)(\nabla^\perp U\cdot V) 
+
\int_{\partial\Om_\eps}
\partial_{s_\eps}((\nabla U\cdot V)^2)\frac{V\cdot\tau_\eps}{V\cdot\nu_\eps}
\nonumber\\
&=&
2
\int_{\Om_\eps}
\nabla^\perp\hat{U}'\cdot\nabla(\nabla^\perp U\cdot V) 
-
\int_{\partial\Om_\eps}
|\nabla U|^2 (V\cdot\nu_\eps)^2\partial_{s_\eps}\left(\frac{V\cdot\tau_\eps}{V\cdot\nu_\eps}\right).
\label{e:J22}
\end{eqnarray}
{Note that the second integral in \eqref{e:J22} is well-defined because
$V|_{\partial\Om_\eps}\in W^{1,\infty}(\partial\Om_\eps;\mathbb R^2)$.}
For $J_3(\eps)$, we decompose $D^2U$ in two terms, one having $\partial_{\nu_\eps\nu_\eps}U$ and $\partial_{\tau_\eps\tau_\eps}U$ derivatives, and the other the mixed derivative $\partial_{\tau_\eps}(\partial_{\nu_\eps} U)$:
\begin{eqnarray}
J_3(\eps)
&=&
-
\int_{\partial\Om_\eps}
(V\cdot \nu_\eps)^2\partial_{\nu_\eps} U\partial_{\nu_\eps\nu_\eps}U
+
(V\cdot \tau_\eps)^2\partial_{\nu_\eps} U\partial_{\tau_\eps\tau_\eps}U
-
\int_{\partial\Om_\eps}
(V\cdot \tau_\eps)(V\cdot\nu_\eps)\partial_{\tau_\eps}((\partial_{\nu_\eps} U)^2)
\nonumber\\
&=&
J_{31}(\eps)+J_{32}(\eps).
\label{e:J3}
\end{eqnarray}
In combination with (ii) (\ref{e:DnnU}), we get
\begin{eqnarray}
J_{31}(\eps)
&=&
\int_{\partial\Om_\eps}
f(\partial_{\nu_\eps} U)(V\cdot \nu_\eps)^2 
+ 
{\cal H}_\eps((V\cdot \nu_\eps)^2-(V\cdot\tau_\eps)^2))(\partial_{\nu_\eps} U)^2.
\label{e:J31}
\end{eqnarray}
For $J_{32}(\eps)$  we have
\begin{eqnarray}
J_{32}(\eps)
&=&
-
\int_{\partial\Om_\eps}
(V\cdot \tau_\eps)(V\cdot\nu_\eps)\partial_{s_\eps}((\partial_{\nu_\eps} U)^2)
=
\int_{\partial\Om_\eps}
(\partial_{\nu_\eps} U)^2
\partial_{s_\eps}((V\cdot \tau_\eps)(V\cdot\nu_\eps)).
\label{e:J32}
\end{eqnarray}
{
Then \eqref{e:J22}, \eqref{e:J31}, \eqref{e:J32} give
\begin{eqnarray*}
J_{22}+J_3
&=&
2
\int_{\Om_\eps}
\nabla^\perp\hat{U}'\cdot\nabla(\nabla^\perp U\cdot V) 
\nonumber\\
&&+
\int_{\partial\Om_\eps}
|\nabla U|^2
\Big(
-\partial_{s_\eps}(V\cdot\tau_\eps)(V\cdot\nu_\eps)
+(V\cdot\tau_\eps)\partial_{s_\eps}(V\cdot\nu_\eps)
\nonumber\\
&&
\hspace*{22mm}
+ 
{\cal H}_\eps((V\cdot \nu_\eps)^2-(V\cdot\tau_\eps)^2))
\nonumber\\
&&
\hspace*{22mm}
+\partial_{s_\eps}(V\cdot\tau_\eps)(V\cdot\nu_\eps)
+(V\cdot\tau_\eps)\partial_{s_\eps}(V\cdot\nu_\eps)
\Big)
\nonumber\\
&&+
\int_{\partial\Om_\eps}
f(\partial_{\nu_\eps} U)(V\cdot \nu_\eps)^2.
\end{eqnarray*}
}
{
So
\begin{eqnarray}
J_{22}+J_3
&=&
2
\int_{\Om_\eps}
\nabla^\perp\hat{U}'\cdot\nabla(\nabla^\perp U\cdot V) 
\nonumber\\
&&+
\int_{\partial\Om_\eps}
\Big(
f(\partial_{\nu_\eps} U)(V\cdot \nu_\eps)^2
+
{\mathcal H}_\eps|\nabla U|^2{|V|^2}
+
2|\nabla U|^2
(|V\cdot\tau_\eps)(\partial_{s_\eps} V\cdot\nu_\eps)
\Big)
.
\label{e:J22+J3}
\end{eqnarray}
}
We denote $K_{i}(\eps)$ each term in the last formula, for $i\in\llbracket1,4\rrbracket$.
{We will show that each of $K_1(\eps)$, $K_2(\eps)$ and $K_4(\eps)$ has a limit as $\eps\to0$ and will identify the limits.
{ As about  $K_3(\eps)$, {we will only use that it is nonnegative (see Remark \ref{rk:K3} and the appendix for more details about $\lim_{\eps\to 0}K_{3}(\eps)$}).}

To identify the limits, we will write first the boundary integrals as integrals in $\Om_\eps$, then we will pass in limit and finally we will write the terms as integrals on $\partial\Om$.
{For $K_1(\eps)$ we have (using $\widehat{U}'=0$ on $\partial\Om$, 
$\nabla\cdot\left(\nabla^\perp\right)=0$ { and $U\in H^2(\Om)$})}
\begin{eqnarray}
K_1(\eps)
&\xrightarrow[\eps\to0]{}&
2
\int_{\Om}
\nabla^\perp\hat{U}'\cdot\nabla(\nabla^\perp U\cdot V) 
=
-2
\int_{\Om}
\nabla\hat{U}'\cdot\nabla^\perp(\nabla^\perp U\cdot V) 
=
0.
\label{e:K1}
\end{eqnarray}
For $K_2(\eps)$ we have {(using again $U\in H^2(\Om)$)}
\begin{eqnarray}
K_2(\eps)
=
\int_{\partial\Om_\eps}
\nu_\eps\cdot(fV(\nabla U\cdot V))
 &=&
\int_{\Om_\eps}
\nabla\cdot(fV(\nabla U\cdot V))\nonumber\\
&\xrightarrow[\eps\to0]{}&
\int_{\Om}
\nabla\cdot(fV(\nabla U\cdot V))
=
\int_{\partial\Om}
f(\partial_{\nu_\eps} U)(V\cdot \nu_\eps)^2.
\label{e:K2}
\end{eqnarray}
{For $K_4(\eps)$ we will prove
\begin{equation}
\lim_{\eps\to0}K_4(\eps)
=
2
\int_{\partial\Om}
|\nabla U|^2
(V\cdot\tau)(\partial_{s} V\cdot\nu).
\label{e:K4}
\end{equation}
Using the divergence theorem to transform the boundary integral in domain integrals as we did above does not work, because a second order derivative of $V$ appears. 
Rather, we proceed as follows.
Let $\partial\Om=\{(r(\theta),\theta),\; \theta\in[0,2\pi]\}$
and $\partial\Omega_\eps=\{(r_\eps(\theta),\theta),\; \theta\in[0,2\pi]\}$. Note that from 
$U(r_\eps(\theta),\theta)=\eps$ on $\partial\Om_\eps$, if $(\varrho,\theta)$ are the polar coordinates we get
\[
 \partial_\varrho U r_\eps' + \partial_\theta U
 =0,\quad so\quad
 r_\eps' = -\frac{\partial_\theta U}{\partial_\varrho U}
 \quad on\quad \partial\Om_\eps.
\]
For $|\eps|$ small all $\Omega_\eps$ are convex and contain a {fixed} small ball. Then necessarily 
$r'_\eps$ is bounded uniformly in $\eps$.
It implies
\begin{eqnarray}
\sqrt{r_\eps^2+(r'_\eps)^2}
&=&
\sqrt{r_\eps^2+\left(\frac{\partial_\theta U}{\partial_\varrho U}\right)^2}
\\
&=&
\frac{{r_\eps}}{|\partial_\varrho U|}
\sqrt{(\partial_\varrho U)^2 +\frac{1}{r_\eps^2}(\partial_\theta U)^2}
\nonumber\\
&=&
r_\eps \frac{|\nabla U|}{|\partial_\varrho U|}
\quad
on
\quad\partial\Om_\eps.
\label{e:ds_eps}
\end{eqnarray}
Using the formulas for $\nu_\eps$ and $\tau_\eps$ in terms of $\nabla U$, 
and then changing the variable $ds_\eps=\sqrt{(r_\eps)^2 + (r'_\eps)^2}d\theta$ 
we get
\begin{eqnarray}
\hspace*{-7mm}
K_4(\eps)
&=&
{-}2\int_{\partial\Om_\eps}
|\nabla U|^2
\left(
V\cdot\frac{\nabla^\perp U}{|\nabla U|}
\right)
\left(\frac{\nabla U}{|\nabla U|}\cdot\nabla V\cdot\frac{\nabla^\perp U}{|\nabla U|}\right)
ds_\eps
\nonumber\\
\hspace*{-7mm}
&=&
{-}2\int_{\partial\Om_\eps}
\left(
(V\cdot\frac{\nabla^\perp U}{|\nabla U|}
\right)
\left(\nabla U\cdot\nabla V\cdot\nabla^\perp U\right)
ds_\eps
\nonumber\\
\hspace*{-7mm}
&=&
{-}2\int_{0}^{2\pi}
(V(r_\eps,\theta)\cdot\nabla^\perp U(r_\eps,\theta))
(\nabla U(r_\eps,\theta)
\cdot
\nabla V(r_\eps,\theta)
\cdot
\nabla^\perp U(r_\eps,\theta))
\frac{r_\eps}{|\partial_\varrho U(r_\eps,\theta)|} 
d\theta,
\nonumber
\end{eqnarray}
where $r_\eps=r_\eps(\theta)$.
A similar formula holds for $K_4(0)$ {with $r$ instead of of $r_\eps$}.
Using the fact that $U\in W^{1,\infty}(\Om)\cap H^2(\Om)$ and Lemma \ref{l:Hv} we get
\begin{eqnarray}
&&
\hspace{-15mm}\frac{{K_4(\eps)-K_4(0)}}{2}
\nonumber\\
&=&
\int_{0}^{2\pi}
\Big(
(V(r,\theta)\cdot\nabla^\perp U(r,\theta))
(\nabla U(r,\theta)
\cdot
\nabla V(r,\theta)
\cdot
\nabla^\perp U(r,\theta))
\frac{r}{|\partial_\varrho U(r,\theta)|}
-
\nonumber\\
&&
\hspace*{11mm}
(V(r_\eps,\theta)\cdot\nabla^\perp U(r_\eps,\theta))
(\nabla U(r_\eps,\theta)
\cdot
\nabla V(r_\eps,\theta)
\cdot
\nabla^\perp U(r_\eps,\theta))
\frac{r_\eps}{|\partial_\varrho U(r_\eps,\theta)|} 
\Big)
d\theta
\nonumber\\
&=&
\int_0^{2\pi}
\int_{r_\eps(\theta)}^{r(\theta)}
\partial_\varrho
\Big(
(V(\varrho,\cdot)\cdot\nabla^\perp U(\varrho,\cdot))
(\nabla U(\varrho,\cdot)
\cdot
\nabla V(\varrho,\cdot)
\cdot
\nabla^\perp U(\varrho,\cdot))
\frac{\varrho}{|\partial_\varrho U(\varrho,\cdot)|}
\Big)
d\varrho d\theta
\nonumber\\
&=&
\int_{\Om\backslash\Om_\eps}
\frac{1}{\varrho}
\partial_\varrho
\Big(
(V(\varrho,\theta)\cdot\nabla^\perp U(\varrho,\theta))
(\nabla U(\varrho,\theta)
\cdot
\nabla V(\varrho,\theta)
\cdot
\nabla^\perp U(\varrho,\theta))
\frac{\varrho}{|\partial_\varrho U(\varrho,\theta)|}
\Big){dx}
\nonumber\\
&\xrightarrow[\eps\to0]{}&
0,
\label{e:K4-K4->0}
\end{eqnarray}
which proves \eqref{e:K4}.
The limit \eqref{e:K4-K4->0} holds because
$r_\eps\to r$ in $L^\infty(\T)$ and the function under the integral is in $L^2(\Om)$
($|\nabla U|/|\partial_\varrho U|$ is bounded - see \eqref{e:ds_eps},
$U\in W^{1,\infty}(\Om)\cap H^2(\Om)$, $V\in W^{1,\infty}(\R^2;\mathbb R^2)$, and $\partial_\varrho \nabla V\in L^\infty(\Om)$ by Lemma \ref{l:Hv}).
}

Replacing \eqref{e:K1}-\eqref{e:K4} in \eqref{e:J22+J3}, and then replacing \eqref{e:J22+J3}, \eqref{e:J21} and
\eqref{e:J1} in \eqref{e:e_D''=..+J} completes the proof.
\hfill$\Box$

~\par

\subsection{Shape derivatives of the first eigenvalue}

For $\lambda_1$ we have the following result, very similar to Theorem \ref{th:e_f''(0)=}.
\begin{theorem}\label{th:l_1''(0)=}
Under the same assumptions as in Theorem \ref{th:e_f''(0)=},
{$\lambda_1$ is twice shape differentiable at $\Om$} and
{\begin{eqnarray}
\hspace*{-9mm}
\lambda_1''(\Om)(V,V)
&\geq&
\!\!\int_{\Om}
|\nabla U_{1}'|^2 
\nonumber\\
&&+
\!\!
\int_{\partial\Om}
\Big(
{\frac{1}{2}\lambda_{1}(\Om) (U_{1}{\partial_{\nu} U_{1})}(V\cdot\nu)^2}+
(\partial_{\nu}U_{1})^2 (V\cdot\tau)(\partial_{s} V\cdot \nu)
\Big)ds,
\label{e:l''(0),main}
\end{eqnarray}}
where
$U_{1}\in H^1_0(\Om)$ {is solution to \eqref{e:l1} such that 
$\int_{\Om}U_{1}^2dx=1$} and $U_{1}'\in H^1(\Om)$ solves in a variational sense
\begin{eqnarray}
-\Delta U_{1}' &=&\lambda_1'{(\Om)(V)}U_{1}+\lambda_1(\Om) U_{1}'\;\; \mbox{in}\;\;\Om,\quad
U_{1}' = -\nabla U_{1}\cdot V\;\;\mbox{on}\;\;\partial\Om.
\label{e:l',U'}
\end{eqnarray}
\end{theorem}
{As for the proof of Theorem \ref{th:e_f''(0)=}, the starting point for proving this result is the following formula that one can find in \cite[Section 4.3.3]{LNP1}
\[\lambda_1''(\Om)(V,V)
=
\lambda_1\int_{\Om}|U_{1}'|^2
-
\int_{\Om}
U_{1}\Delta U_{1}'',
\]
where 
\begin{eqnarray}
U_{1}'&=&\hat{U}_{1}'-V\cdot U_{1},
\label{e:U',l}
\\
U''_{1}&=&\hat{U}_{1}''-2V\cdot\nabla U_{1} -V\cdot D^2U_{1}\cdot V.
\label{e:U'',l}
\end{eqnarray}
Furthermore, {$U_1>0$ { in $\Om$} and if {we denote}
$\Omega_\eps=\{x\in\Om,\;\; U_1(x)>\eps\}$}, we have
\begin{equation}
{\lambda_1''(\Om)(V,V)}
=
\lim_{\eps\to0}
\int_{\partial\Om_\eps}
U_{1}''(\nabla U_{1}\cdot\nu_\eps).
\label{e:l=lim}
\end{equation}

{
The rest of the proof is exactly similar to the case of $E_f''(\Om)$ from Theorem \ref{th:e_f''(0)=}: first we note that \eqref{e:l=lim} differs from \eqref{e:e_f''}
only by the fact that it has not the term with $f$ and {involves $U_1$} instead of $U$ ;
next, we repeat the computations as in Theorem \ref{th:e_f''(0)=} 
with the only exception that instead of $-\Delta U=f$ we have 
$-\Delta U_{1}=\lambda_1 U_{1}${; last, it is classical that $U_{1}>0$ in $\Om$ and that the level sets 
$\{U_{1}>\eps\}$ are convex by \cite{BL1977}.}
Therefore, we do not reproduce the proof. 
{See also the appendix for a complete formula for $\lambda_{1}''(\Om)$ when $\Om$ is convex.}\\}

{We take the opportunity to notice that, for the proof of Theorem \ref{th:l_1''(0)->convex}, we will need also the following classical result for $\lambda_1'(\Om)$, see for example \cite{HP05Var}}
\begin{eqnarray}
\hspace*{-9mm}
{\lambda_1'(\Om)(V)}
&=&
{-\int_{\Om}
|\nabla U_{1}|^2(V\cdot\nu)ds,}
\label{e:l'(0),main}
\end{eqnarray}

\section{Convexity preserving perturbations}\label{sect:perturbation}
{In this section we will consider vector fields of the form $V=v z$ with 
$v\in W^{1,\infty}(\partial\Omega)$ and $z$ {is a constant vector in }$\R^2$}. {These vector fields can easily be extended in $\R^2$, for example using Lemma \ref{l:Hv} { we denote $V=\bar{v}z$ the extension} (though in this Section, any other Lipschitz extension also works).}
If we assume $\Omega$ is a convex polygon and consider the vector field
 $V=v z$, with $v$ the hat function associated to three (consecutive) vertices 
 $P_{1},P_{2},P_{3}\in\partial\Om$ (see Definition \ref{d:hat}),
then the shape $\Omega_t=(I+tV)(\Omega)$ is clearly still a convex polygon for $t$ small, {whatever the vector $z$ is.}

{We {will} show that if $\Om$ is convex but not assumed to be a polygon, one can still construct perturbations of the form $\Omega_t=(I+tV)(\Omega)$ that remain convex, where $V=vz$ with $v$ close to a hat function, and $z$ being suitably chosen. As for our purpose it is enough to work locally, we will see 
$\Om$ as a local graph of a convex function {$u$ on $I=(0,\sigma)$}, which leads us to build local perturbation of convex functions on the real line. 
\\

In a similar manner to Definition \ref{def:convex} we introduce this definition:
\begin{definition}\label{def:convexgraphe}
Let {$I=(0,\sigma)$, $\sigma>0$, and} $u:I\to\R$ be a convex function in the interval $I$ and $v:I\to\R$. We say that $v$ preserves the convexity of $u$ if there exists $t_{0}>0$ such that $u+tv$ is convex for every $t\in[-t_{0},t_{0}]$.
\end{definition}
{
In the case where $\Om$ is locally the graph of $u:I\to\R$ and $v:I\to\R$ preserves the convexity of $u$ and 
is compactly supported in $I$}, then {$V={\bar{v}}z$} with $z=(0,1)$ preserves the convexity of 
$\Om$ in the sense of Definition \ref{def:convex}.
}

{\begin{proposition}\label{p:varphi}
Let $u:I\to\R$ be convex on $I=(0,\sigma)$, $\sigma>0$, $u\in W^{1,\infty}(0,\sigma)$ {such that} $u''\geq0$ and 
$\{x_{1},x_2,x_3\}\subset{\rm supp}(u'')$ with $0<x_1<x_2<x_3<\sigma$. 
\begin{enumerate}[label=(\roman*)]
\item 
There exists a sequence $(\varphi_k)$ in $H^1_0(0,\sigma)$ preserving the convexity of $u$, 
{compactly supported in $(0,\sigma)$} and converging in $H^1(0,\sigma)$ to the hat function $\varphi$ associated with the points $x_1,x_2,x_3$.
\item
Assume $x_2$ is an accumulation point of ${\rm supp}(u'')$.
Then there exist two sequences $(\varphi_k^-)$ and  $(\varphi_k^+)$ in $H^1_0(0,\sigma)$ 
preserving the convexity of $u$, {compactly supported in $(0,\sigma)$}  and converging in $L^1(0,\sigma)$ respectively 
to $\varphi\1_{(x_1,x_2)}$, $\varphi\1_{(x_2,x_3)}$,
 where $\varphi$ is the hat function in i) above.
 \end{enumerate}
\end{proposition}}
\noindent
{\bf Proof}.
Let $0<a<b<c<d<\sigma$ with $(b,c)\cap{\rm supp}(u'')\neq\emptyset$. Define $w=w(a,b,c,d)$ by
\[
w''=u''\1_{(b,c)}\;\; in\;\; {\cal D}'({a,d}),\qquad w\in H^1_0(a,d).
\]
It is classical that $w\in W^{1,\infty}(a,d)$ and $w$ is convex. {As $\supp(u'')\cap(b,c)\neq\emptyset$, $w'(a^+)$ and $w(d^-)$ cannot be 0}. {From the convexity of $w$ and the fact that $w$ is linear in
$(a,b)\cup(c,d)$, the graph of $w$ is under the piece-wise linear curve joining the points $(a,0)$, $(c,w(a^+)(c-a))$ and
$(d,0)$, {and above the one joining $(a,0), (b,w'(d^-)(b-d)), (d,0)$}, therefore}
we have
\begin{equation}
\frac{b-a}{d-b}\leq \left|\frac{w'(d^-)}{w'(a^+)}\right|\leq \frac{c-a}{d-c}.
\label{e:est(w(a,d))}
\end{equation}
\\
{Proof of $(i)$:} 
Let
\begin{eqnarray*}
v_{1,k}&=&
C_{1,k}w(x_1-2\delta_k,x_1-\delta_k,x_1+\delta_k,x_3+2\delta_k)
{=
:
C_{1,k}w(a_k,b_{1,k},c_{1,k},d_k),}
\\
v_{2,k}&=& C_{2,k}w(x_1-2\delta_k,x_2-\delta_k,x_2+\delta_k,x_3+2\delta_k)
{=
:
C_{2,k}w(a_k,b_{2,k},c_{2,k},d_k),}
\\
v_{3,k}&=&C_{3,k}w(x_1-2\delta_k,x_3-\delta_k,x_3+\delta_k,x_3+2\delta_k)
{=
:
C_{3,k}w(a_k,b_{3,k},c_{3,k},d_k)}.
\end{eqnarray*}
where {the sequence $(\delta_k)$ is positive and converges to zero} and 
$C_{1,k}$, $C_{2,k}$, $C_{3,k}$ such that
{{$v_{1,k}'(a_k^+)=v_{2,k}'(a_k^+)=-1$,
$v_{3,k}'(d_k^-)=1$}}.

{We look for $\lambda_{1,k}$, $\lambda_{3,k}$ such that}
$v_k=\lambda_{1,k} v_{1,k} + v_{2,k} + \lambda_{3,k} v_{3,k}$
satisfies 
$v'_k(a_k^+)=v_k'(b_k^-)=0$}. In other words, $\lambda_{1,k}, \lambda_{3,k}$ must solve
{
\begin{equation}
\left\{
\begin{array}{rcl}
\lambda_{1,k} v_{1,k}'(a_k^+) + \lambda_{3,k} v_{3,k}'(a_k^+)&=&-{v_{2,k}'}(a_k^+),
\\
\lambda_{1,k} v_{1,k}'(d_k^-) + \lambda_{3,k} {v_{3,k}'}(d_k^-)&=&-v_{2,k}'(d_k^-).
\end{array}
\right.
\label{e:l1,l3}
\end{equation}}
Using {\eqref{e:est(w(a,d))}}  we get
{
\begin{equation}
v_{1,k}'(d_k^-)=o(1),
\quad
v_{2,k}'(d_k^-)=\frac{x_2-x_1}{x_3-x_2}+o(1),
\quad
v_{3,k}'(a_k^+)=o(1),
\label{e:v1k',v2k',v3k'}
\end{equation}}
where ${\displaystyle \lim_{\delta_k\to0}o(1)=0}$.
Then \eqref{e:l1,l3} is equivalent to 
\[
\left\{
\begin{array}{rcl}
-\lambda_{1,k} + \lambda_{3,k} o(1)&=&1,
\vspace*{2mm}
\\
\lambda_{1,k} o(1) + \lambda_{3,k}&=&-{\displaystyle \frac{x_2-x_1}{x_3-x_1}}+o(1)
\end{array}
\right.
\Rightarrow
\left\{
\begin{array}{rcl}
\lambda_{1,k} &=&-1+o(1),
\vspace*{2mm}
\\
\lambda_{3,k}&=&-{\displaystyle \frac{x_2-x_1}{x_3-x_1}}+o(1).
\end{array}
\right.
\]
So {$\lambda_{1,k}$ and $\lambda_{3,k}$ exist for $k$ large enough. Moreover, $v_{k}$ preserves the convexity of $u$: indeed with the choice of $\lambda_{1,k},\lambda_{3,k}$, we have 
${\displaystyle v_{k}''
=
\sum_{i=1}^3{\lambda_{i,k}C_{i,k}}u''\mathbb{1}_{(x_{i}-\delta_{k},x_{i}+\delta_{k})}\textrm{ in }\mathcal{D}'(0,\sigma)}$}.

{Besides}, {defining $\lambda_{2,k}=1$}, {for $i=1,2,3$ we have 
$$
\int_0^\sigma \lambda_{i,k}v_{i,k}''
=
\lambda_{i,k}
(v_{i,k}'(c_{i,k}^-) - v_{i,k}'(b_{i,k}^+)),
$$
where $v_{i,k}'(x^-)$, resp. $v_{i,k}'(x^+)$, denotes the left, resp. right, limit at $x$ of $v_{i,k}'$,
and then  by using \eqref{e:v1k',v2k',v3k'} it is easy to show that the limits 
${\displaystyle\lim_{k\to\infty}\int_0^\sigma \lambda_{i,k}v_{i,k}''}$ exist for $i=1,2,3$}.
Therefore {by Lemma \ref{l:f_ik->Ci} below} $v_k$ extended by zero in $(0,\sigma)$, converges in $H^1(0,\sigma)$ to a linear combination of the hat functions associated to $0$, $x_i$, $\sigma$, $i\in\{1,2,3\}$.
It implies ${\displaystyle \varphi_k=\frac{v_k}{{v_{k}}(x_2)}}$ is the required sequence because 
{$\varphi_{k}(0)=\varphi_{k}(x_1)=\varphi_{k}(x_3)=\varphi_{k}(\sigma)=0$} and
{$v_k(x_2)=x_1-x_2+o(1)$}.

\noindent
{Proof of $(ii)$:} Let us first construct the sequence $(\varphi_k^-)$.
As $x_2$ is an accumulation point of ${\rm supp}(u'')$ there exists $(x_{2,k})$ converging to $x_2$  and 
$(\delta_k)$ positive and converging to zero such that 
\[
(x_{2,k}-\delta_k,x_{2,k})\cap{\rm supp}(u'')\neq\emptyset,
\quad
(x_{2,k},x_{2,k}+\delta_k)\cap{\rm supp}(u'')\neq\emptyset.
\]
We construct $\varphi_k^-$ similarly as we did for $\varphi_k$ in $(i)$ above.
Namely, let 
$v_{1,k}=C_{1,k}w(x_1-2\delta_k, x_1-\delta_k,x_1+\delta_k,x_2+2\delta_k)$,
$v_{2,k}=C_{2,k}w(x_1-2\delta_k, x_{2,k}-\delta_k,x_{2,k},x_2+2\delta_k)$,
$v_{3,k}=C_{3,k}w(x_1-2\delta_k, x_{2,k},x_{2,k}+\delta_k, x_2+2\delta_k)$,
with $C_{1,k}$, $C_{2,k}$ and $C_{3,k}$ such that 
$v'_{1,k}((x_1-2\delta_k)^+)=v'_{2,k}((x_1-2\delta_k)^+)=-1$, $v_{3,k}'((x_2+2\delta_k)^-)=1$.

Next take $v_k=\lambda_{1,k}v_{1,k}+v_{2,k}+\lambda_{3,k}v_{3,k}$ such that
$v'_k((x_1-2\delta_k)^+)=v_k'((x_2+2\delta_k)^-)=0$ { so that again it preserves the convexity of $u$}. Then $\lambda_{1,k}$ and $\lambda_{3,k}$ solve \eqref{e:l1,l3}
with $x_2+2\delta_k$ instead of $x_3+2\delta_k$.
Using \eqref{e:est(w(a,d))}  we get
\[
v_{1,k}'((x_2+2\delta_k)^-)=o(1),
\quad
v_{2,k}'((x_1-2\delta_k)^+)=o(1),
\quad
v_{3,k}'((x_1-2\delta_k)^+)=o(1).
\]
Then the solution of the system above is 
$\lambda_{1,k} =-1+o(1)$,
$\lambda_{3,k}=-1+o(1)$.

{Let $\lambda_{2,k}=1$ and extend $v_k$ and $v_{i,k}$ in $(0,\sigma)$ by zero. Then we have
\[
v_k''
=
\sum_{i=1}^3 \lambda_{i,k} C_{i,k}u''\1_{(b_{i,k},c_{i,k})}\;\;\; in\;\; {\cal D}'(0,\sigma),
\]
where $(a_k,b_{i,k},c_{i,k},d_k)$ are the points in the definition of $w$ associated to $v_{i,k}$. 
As $u$ is differentiable in a set $I_d\subset I$, $|I_d|=|I|$, it implies that $v_k$ is differentiable in
$I_d$, for all $k$ and regardless the choice of $(\delta_k)$.
We can choose $(\delta_k)$ such that furthermore we have $x_1+\delta_k\in I_d$ for all $k$, 
which implies that $v_k'(x_1+\delta_k)$ exists for all $k$.
}

{From boundary conditions of $v_k$
on $x_1-2\delta_k$ and $x_2+2\delta_k$, and as $v_k$ is 
concave in $[x_1-2\delta_k,x_1+\delta_k]\cup[x_{2,k},x_2+2\delta_k]$, convex in $[x_1+\delta_k,x_{2,k}]$
and linear in $[x_1+\delta_k,x_{2,k}-\delta_k]$ we have
\begin{eqnarray*}
\|v_k\|_{L^\infty(x_1-2\delta_k,x_1+\delta_k)}
&\leq&
 2\delta_k|v_k'(x_1+\delta_k)|,
\\
v_k(x)
&=&
v_k'(x_1+\delta_k)(x-x_1-\delta)+v_k(x_1+\delta_k)\;\; \mbox{in}\;\;  [x_1+\delta_{k},x_{2,k}-\delta_k]
\\
\|v_k\|_{L^\infty(x_1-2\delta_k,x_2+2\delta_k)}
&\leq&
(x_{2,k}-x_1+\delta)|v_k'(x_1+\delta_k)|.
\end{eqnarray*}
Necessarily $v_{k}'(x_{1}+\delta_{k})\neq 0$ {because otherwise $v_k\equiv0$, which is impossible,} 
and then the required sequence is given by 
${\displaystyle \varphi_k^-=\frac{1}{x_2-x_1}\frac{v_k}{v_k'(x_1+\delta_k)}}$.}

For the other sequence we take $\varphi_k^+=\varphi_k-\varphi_k^-$.
\hfill$\Box$

\begin{lemma}\label{l:f_ik->Ci}
Let $0<b_{i,k}<c_{i,k}<\sigma$, $i=1,\ldots,m$, $k\in\N$, $\lim_{k\to\infty}b_{i,k}=\lim_{k\to\infty}c_{i,k}=x_i\in(0,\sigma)$,
with all $x_i$ different,
$v_k\in H^1_0(0,\sigma)$, $v_k''=\sum_{1\leq i\leq m}f_{i,k}$, 
with $f_{i,k}\in H^{-1}(0,\sigma)$, $f_{i,k}$ unsigned, 
${\rm supp}(f_{i,k})\subset(b_{i,k},c_{i,k})$. 
If  $\lim_{k\to\infty}\int_0^\sigma f_{i,k}=C_i$ with certain $C_i\in\R$, for all $i$,
then $f_{i,k}$ converges to $C_i\delta(x-x_i)$, and $v_k$ converges in $H^1(0,\sigma)$ to $v$, solution 
of $v''=\sum_{1\leq i\leq m}C_i\delta(x-x_i)$. Hence, $v$ a linear combination of hat functions associated to 
$\{0,x_i,\sigma\}$, $i=1,\ldots,m$.
\end{lemma}
\begin{proof}
First we note that $f_{i,k}\to C_i\delta(x-x_i)$ strongly in $H^{-1}(0,\sigma)$. Indeed, for all 
$\varphi,\, \eta_i\in H^1(0,\sigma)$, with $\eta_i=1$ in  $[x_i-\eps,x_i+\eps]$, 
$\eta_i=0$ in  $(0,\sigma)\backslash[x_i-2\eps,x_i+2\eps]$, for $k$ large we have
\begin{eqnarray*}
\left|\int_0^\sigma(f_{i,k}-C_i\delta(x-x_i))\varphi dx\right|
&=&
\left|\int_0^\sigma(f_{i,k}-C_i\delta(x-x_i))\varphi\eta_i dx\right|
\\
&\leq&
|\varphi(x_i)|\left|C_i-\int_0^\sigma f_{i,k}dx\right|
+
\|\eta_i(\varphi-\varphi(x_i))\|_{C^0(0,\sigma)}\|f_{i,k}\|_{C'_0([0,\sigma])}
\\
&\leq&
\left(\left|C_i-\int_0^\sigma f_{i,k}dx\right|
+
2\eps^{1/2}\|f_{i,k}\|_{C'_0([0,\sigma])}
\right)
\|\varphi\|_{H^1(0,\sigma)},
\end{eqnarray*} 
which shows the convergence of $f_{i,k}$.
Then the $H^1(0,\sigma)$ convergence of $v_k$ to $v$ follows from the continuity of the Dirichlet problem.
\end{proof}

{\section{{Convexity of $E_{f}$ and $\lambda_1$ and application to their maximization}\label{ss:convexity-e''}}}

{In this section, we will prove Theorem \ref{th:main} for problem {\eqref{e:Pb1} and \eqref{e:Pb2}. We will consider in detail the proof for the problem \eqref{e:Pb1}, {the proof for \eqref{e:Pb2} being very similar}}. 
As mentioned in the introduction, this result is a consequence of Theorem \ref{th:e_f''(0)->convex} which is a weak convexity property of the functional $E_{f}$. So we start this section by proving the latter.\\

\subsection{{Proof of Theorems \ref{th:e_f''(0)->convex} and \ref{th:l_1''(0)->convex}}}
Let $\Om_{0}$, {$\Gamma_{0}$}, ${\sigma_{0}}$, $u_0$ be as in Theorem \ref{th:e_f''(0)->convex}. { Let also $z=(0,1)$ and $V=\bar{v} z$ with $\bar{v}\in W^{1,\infty}(\R^2)$ the extension of $v\in W^{1,\infty}(\partial\Om_{0})$ given by Lemma \ref{l:Hv}.}
It follows from 
Theorem \ref{th:e_f''(0)=} that 
\begin{eqnarray}
\hspace*{-10mm}
E_{f}''(\Om_{0})(V,V)
 \geq 
\int_{\Omega_{0}}
|\nabla U'_0|^2 
&+&\frac{1}{2}
\int_{\partial\Om_{0}}
\!\!\!\!
f(\partial_{\nu_0} U_0)(z\cdot \nu_0)^2
v^2
ds_0
\nonumber
\\
\hspace*{-10mm}
&+&
\int_{\partial\Omega_{0}}
(\partial_{\nu_{0}} U_0)^2(z\cdot\tau_0)
(z\cdot\nu_0)v\partial_{s_0} v ds_0.
\label{eq:Ef''bis} 
\end{eqnarray}
In order to prove Theorem \ref{th:e_f''(0)->convex}, it remains {to show that the right hand side in \eqref{eq:Ef''bis} is positive for some specific choice of  
$v$}.

As in the assumptions of Theorem  \ref{th:e_f''(0)->convex}, let {$\varphi\in H^1_0(0,{\sigma_{0}})$} be the hat function associated to the triplet of points 
$x_{1}<x_{2}<x_{3}$ in $(0,{\sigma_{0}})$,
and we take {$v\in W^{1,\infty}(\partial\Om_0)$ defined by:
$$ {v(p)=\varphi(x),\;\; \forall p=(x,u_0(x))\textrm{ with } x\in (0,\sigma_0),\;\; and\;\; 
    v(p)=0,\;\ \forall p\in\partial\Om_{0}\setminus\Gamma_{0}}.
    $$
{Note that in all this Section, we will work in the parametrization of $\Gamma_{0}$ given by $u_{0}$: for example, $(\partial_{\nu_{0}} U_0)^2$ may refer sometimes to $x\in(0,\sigma_{0})\mapsto (\partial_{\nu_{0}} U_0)^2(x,u_{0}(x))$. In order to avoid confusion, $x$ will always denote an element of $(0,\sigma_{0})$.}

}

{The next lemma allows to study the sign of the last term in \eqref{eq:Ef''bis} for this choice of $v$. 
}
\begin{lemma}\label{l:...phi'phi>=0}
Under the conditions of Theorem \ref{th:e_f''(0)->convex}, {if $x_{3}$ is small enough} we have
\[
\int_{\partial\Om_0}
(\partial_{\nu_{0}} U_0)^2(z\cdot\tau_0)(z\cdot\nu_0) v\partial_{s_0}v ds_0
\geq0.
\]
\end{lemma}
{\begin{proof}
Note that in the local system of coordinates {where $\Gamma_{0}=\{(x,u_{0}(x)), x\in (0,\sigma_{0})\}$}, we have 
$\nu_0=(u_0',-1)/\sqrt{1+(u_0')^2}$,
$\tau_0=(1,u_0')/\sqrt{1+(u_0')^2}$ 
and $ds_0=\sqrt{1+(u_0')^2}dx$. Therefore
\begin{eqnarray*}
\int_{\partial\Omega_0}
(\partial_{\nu_{0}} U_0)^2(z\cdot\tau_0)(z\cdot\nu_0) v{\partial_{s_0}v} ds_0
&=&
-\int_{x_{1}}^{x_{3}}
(\partial_{\nu_{0}} U_0)^2\frac{u_0'}{1+(u_0')^2}{\varphi\varphi'}dx.
\end{eqnarray*}

From {assumption $(ii)$ in Theorem \ref{th:e_f''(0)->convex}} we have 
\begin{eqnarray}\label{eq:EL1}
\forall i\in\{1,2\}, \;\;\int_{x_{i}}^{x_{i+1}}
(\partial_{\nu_{0}} U_0)^2{\varphi dx}
&=&
2\mu \int_{x_{i}}^{x_{i+1}}{\varphi dx}
=
\mu|x_{i+1}-x_{i}|,
\end{eqnarray}
{From $u_{0}'(0^+)=0$ we deduce that if $x_{3}$ is small enough, we have $|u_{0}'|\leq 1$ in $(0,x_{3})$}. Note also that 
$u_0'$ is increasing and positive, and $r\in[0, 1]\mapsto\frac{r}{1+r^2}$ is increasing,
so this implies
$\frac{u_0'}{1+(u_0')^2}$ is increasing {in $(0,x_{3})$}.
Using the monotonicity of $\frac{u_{0}'}{1+(u_0')^2}$ and \eqref{eq:EL1} we get
\begin{eqnarray*}
\int_{\partial\Omega_0}
(\partial_{\nu_{0}} U_0)^2&&\hspace{-10mm}(z\cdot\tau_0)(z\cdot\nu_0)v\partial_{s_0} vds_0\\
&=&
-
\int_{x_{1}}^{x_2}
(\partial_{\nu_{0}} U_0)^2\frac{u_0'}{1+(u_0')^2}{\varphi  \varphi'dx}
-
\int_{x_2}^{x_{3}}
(\partial_{\nu_{0}} U_0)^2\frac{u_0'}{1+(u_0')^2}{\varphi \varphi'dx}
\\
&=&
-
\frac{1}{|x_2-x_{1}|}
\int_{x_{1}}^{x_2}
(\partial_{\nu_{0}} U_0)^2\frac{u_0'}{1+(u_0')^2}{\varphi dx}
+
\frac{1}{|x_{3}-x_2|}
\int_{x_2}^{x_{3}}
(\partial_{\nu_{0}} U_0)^2\frac{u_0'}{1+(u_0')^2}{\varphi  dx}
\\
&\geq&
-
\left(\frac{u_0'(x_2^-)}{1+(u_0'(x_2^-))^2}\right)
\frac{1}{|x_2-x_{1}|}
\int_{x_{1}}^{x_2}
(\partial_{\nu_{0}} U_0)^2{\varphi  dx}
\\&&\hspace{35mm}+
\left(\frac{u_0'(x_2^+)}{1+(u_0'(x_2^+))^2}
\right)
\frac{1}{|x_2-x_{3}|}
\int_{x_2}^{x_{3}}
(\partial_{\nu_{0}} U_0)^2{\varphi  dx}
\\
&=&
\mu \Big\llbracket \frac{u_0'}{1+(u_0')^2}\Big\rrbracket_{x_2}
\geq
0,
\end{eqnarray*}
{where $\llbracket f\rrbracket_{x_2}:=f(x_{2}^+)-f(x_{2}^-)$. This concludes the proof.}
\end{proof}}

{The following two lemma will help analyzing the second term in \eqref{eq:Ef''bis}.}
\begin{lemma}\label{l:int|DU|phi^2<}
{
Under the conditions of Theorem \ref{th:e_f''(0)->convex}}, {if $x_{3}$ is small enough} we have
\[
\int_{\partial \Om_{0}}|\partial_{\nu_{0}}U_0||V\cdot\nu_0|^2ds_0
\leq 
\frac{4}{\sqrt{\mu}}\int_{\partial \Om_{0}}|\partial_{\nu_{0}}U_0|^2|V\cdot\nu_0|^2ds_0.
\]
\end{lemma}
{
\begin{proof}
{
As in Lemma \ref{l:...phi'phi>=0}, if $x_{3}$ is small enough we have $u_{0}'\leq 1$ in $(0,x_{3})$. 
{Using the parametrization in $x$, we have}
\begin{eqnarray*}
\int_{\partial \Om_{0}}|\partial_{\nu_0}U_0||V\cdot\nu_0|^2ds_0
&=&
\int_{x_1}^{x_3}|\partial_{\nu_0}U_0|\frac{1}{\sqrt{1+(u_0')^2}}\varphi^2 dx
\leq
\int_{x_1}^{x_3}|\partial_{\nu_0}U_0|\varphi^2 dx,
\\
\int_{\partial \Om_{0}}|\partial_{\nu_0}U_0|^2|V\cdot\nu_0|^2ds
&=&
\int_{x_1}^{x_3}|\partial_{\nu_0}U_0|^2\frac{1}{\sqrt{1+(u_0')^2}}\varphi^2 dx
\geq
\frac{1}{\sqrt{2}}
\int_{x_1}^{x_3}|\partial_{\nu_{0}}U_0|^2\varphi^2dx.
\end{eqnarray*}
Then to prove the claim is enough to show
\[
\int_{x_1}^{x_3}|\partial_{\nu_{0}}U_0|\varphi^2
\leq 
2\sqrt{\frac{2}{\mu}}\int_{x_1}^{x_3}|\partial_{\nu_{0}}U_0|^2\varphi^2.
\]
}
As {recalled in Proposition \ref{p:reg(U)}}, $U_0\in H^2(\Om_{0})\cap W^{1,\infty}(\Om_0)$, 
{so it is easy to deduce that $|\nabla U_{0}|^2\in H^1(\Om_{0})$, and therefore 
$(\partial_{\nu_0}U_0)^2=-|\nabla U_{0}|^2\in H^{1/2}(\partial\Om_{0})$. Seen as function on $x\in (0,\sigma_{0})$, we still have $x\mapsto (\partial_{\nu_0}U_0)^2(x,u_{0}(x))$ belongs to $H^{1/2}(0,\sigma_{0})$ as the transformation is Lipschitz.}
Though it is well-known that functions in $H^{1/2}(0,\sigma_{0})$ may not be continuous, they still get some weak continuity property, namely they belong to the space of vanishing mean oscillation, see for example \cite[Example 2]{BN95Deg}. This implies that
$$
\lim_{a\to 0}
\sup_{\omega\subset (0,\sigma_{0}), |\omega|\leq a}
\left\{
\frac{1}{|\omega|}\int_{\omega}\left|(\partial_{\nu_0} U_0)^2-\overline{(\partial_{\nu_0} U_0)^2}_{\omega}
\right|{dx}\right\}=0, \;\;\;
\textrm{ where }
\overline{(\partial_{\nu_0} U_0)^2}_{\omega}
=
\frac{1}{|\omega|}\int_{\omega}(\partial_{\nu_0} U_0)^2{dx}.
$$
From this property we deduce that there exists $a>0$ such that
$$\forall \omega\subset (0,\sigma), |\omega|\leq a,
\frac{1}{|\omega|}\int_{\omega}\left|(\partial_{\nu_0} U_0)^2-\overline{(\partial_{\nu_0} U_0)^2}_{\omega}
\right|{dx}\leq \frac{\mu}{12}.$$
We apply this to {$\omega=[x_{1},x_{3}]$} for $x_{3}<a$. From {assumption
$(ii)$ in Theorem \ref{th:e_f''(0)->convex}} one has 
$\overline{(\partial_{\nu_0} U_0)^2}_{\omega}=\mu$. 
Let denote $J:=\{x\in \omega, \;(\partial_{\nu_0} U_0)^2>\frac{\mu}{2}\}$. 

{We claim that 
\begin{equation}\label{e:Jnc}|J^c|\leq \frac{1}{6}|\omega|.
\end{equation}
Indeed, 
$$
{\frac{\mu}{12}
\geq}
\frac{1}{|\omega|}
\int_{\omega}\left|(\partial_{\nu_0} U_0)^2-\mu\right|\geq \frac{1}{|\omega|}\int_{J^c}\left|(\partial_{\nu_0} U_0)^2-\mu\right|
\geq 
\frac{\mu}{2}
\frac{|J^c|}{|\omega|}
$$}
which proves the claim.

Moreover, by definition we also have
{\begin{equation}\label{e:est1}
\int_{\omega}|\partial_{\nu_0}U_0|\varphi^2
=
\int_{J}|\partial_{\nu_0}U_0|\varphi^2
+
\int_{J^c}|\partial_{\nu_0}U_0|\varphi^2
\leq 
\sqrt{\frac{2}{\mu}}\int_{J}|\partial_{\nu_0}U_0|^2\varphi^2
+
\sqrt{\frac{\mu}{2}}\int_{J^c}\varphi^2
\end{equation}
on one hand, and
\begin{equation}\label{e:est2}
\int_{\omega}|\partial_{\nu_0}U_0|^2\varphi^2
\geq 
\int_{J}|\partial_{\nu_0}U_0|^2 \varphi^2
\geq
\frac{\mu}{2}\int_{J} \varphi^2
\end{equation}
on the other.

{With \eqref{e:Jnc}, we get}
\begin{eqnarray}
\int_{J}\varphi^2
=
\int_{\omega}\varphi^2-\int_{J^c}\varphi^2
&\geq& 
\frac{|\omega|}{3}-|J^c|
\nonumber
\\
&\geq &
|\omega|\left[\frac{1}{3}-\frac{|J^c|}{|\omega|}\right]
\nonumber
\\
&\geq& \frac{1}{6}|\omega|\geq|J^c|\geq \int_{J^c}\varphi ^2
\label{e:est2+}
\end{eqnarray}
This allows to conclude, using \eqref{e:est1}, \eqref{e:est2} and \eqref{e:est2+}
\begin{eqnarray*}
\int_{\omega}
|\partial_{\nu_0}U_0|\varphi^2
&\leq& 
\sqrt{\frac{2}{\mu}}
\int_{J}|\partial_{\nu_0}U_0|^2\varphi^2
+
\sqrt{\frac{\mu}{2}}
\int_{J^{{c}}}\varphi^2
\\
&\leq&
\sqrt{\frac{2}{\mu}}
\int_{J}|\partial_{\nu_0}U_0|^2\varphi^2 
+ 
\sqrt{\frac{\mu}{2}}\cdot\frac{2}{\mu}
\int_{J}|\partial_{\nu_0}U_0|^2\varphi^2\\
&\leq& 
{2\sqrt{\frac{2}{\mu}}}
\int_{\omega}|\partial_{\nu_0}U_0|^2\varphi^2,
\end{eqnarray*}
which proves the lemma.}
\end{proof}
}

{
\begin{lemma}\label{l:L2<H1/2}
Let $\Om_0\subset\R^2$ convex.
There exists $C>0$ such that {for every ${\varepsilon}\in(0,1]$}, for every 
$p_0=(p_{0,x}, p_{0,y})\in\partial\Om_0$ and every $w\in H^{1/2}(\partial\Om_0)$ with
${\rm supp}(w)\subset\partial\Om_0\cap B({p_{0}},{\varepsilon})$, 
{$B(p_0,{\varepsilon}):=\{p\in \R^2,\; |p-p_0|<{\varepsilon}\}$},
we have
\begin{equation}
 \|w\|_{L^2(\partial\Om_0)}
 \leq
 C\sqrt{{\varepsilon}} |w|_{H^{1/2}(\partial\Om_0)}.
\label{e:L2<H1/2}
\end{equation}
\end{lemma}
\begin{proof}
Without loss of generality we assume $p_0=(p_{0,x},0)$.
The idea of the proof is as follows. First, we change the variable and write \eqref{e:L2<H1/2} equivalently on the unit circle $\partial B$, where
$B$ is the unit disk with center in origin. Next, we give an equivalent form for the seminorm $|\cdot|_{H^{1/2}(\partial B)}$ and 
then conclude with a scaling argument.

Note that if $(r,\theta)$ are polar coordinates, the map $(r,\theta)\mapsto (x,y)=(r\cos\theta,r\sin\theta)$ is a 
$C^\infty$ diffeomorphism from  $(0,\infty)\times[-\pi,\pi)$ into $\R^2\backslash(0,0)$. We denote by
$(r(x,y),\theta(x,y))$ its inverse.

As for the proof of Lemma \ref{l:Hv}, we choose an origin inside $\Om_{0}$ and consider 
$r_0\in W^{1,\infty}(\T)$ - the radial function of $\Om_{0}$ such that, in polar coordinates, 
$\partial\Om_0=\{(r,\theta), r\in[0,r_0(\theta)),\;\; \theta\in\T\}$. 
We consider $T(x,y)=r_0(\theta(x,y))\left[\begin{array}{r}x\\y\end{array}\right]$. Note that 
$T$ is $W^{1,\infty}$ in a neighborhood of $\partial B$, 
$T(\partial B)=\partial\Om_0$  and
\begin{eqnarray*}
\partial_x T(x,y)
&=&
\left(-\frac{y}{r^2}\right)r_0'(\theta)\left[\begin{array}{r}x\\y\end{array}\right]
+
r_0(\theta)\left[\begin{array}{r}1\\0\end{array}\right],
\\
\partial_y T(x,y)
&=&
\left(\frac{x}{r^2}\right)r_0'(\theta)\left[\begin{array}{r}x\\y\end{array}\right]
+
r_0(\theta)\left[\begin{array}{r}0\\1\end{array}\right],
\\
{\rm det}[\nabla T(x,y)]
&=&
r_0^2(\theta(x,y))\neq0,
\end{eqnarray*}
which shows that $T$ is a {bi-Lipchitzian} transformation from a neighborhood of $\partial B$ onto a neighborhood of $\partial\Om_0$.
Using the transformation $T$, it is easy to show that the lemma is equivalent to:
there exists $C>0$, such that for every $\varepsilon\in(0,1)$ and for every 
$h \in H^{1/2}_{\varepsilon}(\partial B)
:=
H^{1/2}(\partial B)
\cap
\{h,\; {\rm supp}(h)\subset \partial B\cap B((1,0),\varepsilon)\}$ we have
\begin{equation}
 \|h\|_{L^2(\partial B)}
 \leq
 C\sqrt{{\varepsilon}}
 |h|_{H^{1/2}(\partial B)},\;\,
 \mbox{where}\;\,
 |h|_{H^{1/2}(\partial B)}^2
=
\int_{\partial B}
ds_y
\int_{\partial B}
\frac{|h(x)-h(y)|^2}{|x-y|^2}ds_x
 \label{e:L2<H1/2,1}
\end{equation}
and  $B((1,0),\varepsilon)$ is the ball with center $(1,0)$ and radius
$\varepsilon$.
The proof is based on the facts: 
$\|w\|_{L^2(\partial\Om_0)}$ and $\|h\|_{L^2(\partial B)}$,
$h=T^{-1}w$, are equivalent,
$|w|_{H^{1/2}(\partial\Om_0)}=\|\nabla R_{w,\Om_0}\|_{L^2(\Om_0)}$,
$|h|_{H^{1/2}(\partial B)}=\|\nabla R_{h,B}\|_{L^2(B)}$,
and the norms 
$\|\nabla R_{w,\Om_0}\|_{L^2(\Om_0)}$, 
$\|\nabla R_{h,B}\|_{L^2(B)}$
are equivalent. Here
$R_{f,\Om_0}$ satisfies $-\Delta R_{f,\Om_0}=0$ in $\Om_0$, $R_{f,\Om_0}=f$ on 
$\partial\Om_0$.

Clearly, \eqref{e:L2<H1/2,1} holds for ${\varepsilon}=1$. Indeed, for 
$h\in H^{1/2}_1(\partial B)$ we have $|h|_{H^{1/2}(\partial B)}=\|\nabla R_{h,B}\|_{L^2(B)}$, which is a classical result.
As Poincar\'e inequality holds in $H^1_1(B)$, we have 
$\|\nabla(\cdot)\|_{L^2(B)}$ is a norm on {$H^1_{1}(B)$}. 
{Using the continuity of the trace operator $H^1(B)\mapsto H^{1/2}(\partial B)$}, we  get
\begin{eqnarray}
\|h\|_{L^2(\partial B)}^2
&\leq&
\|h\|_{H^{1/2}(\partial B)}^2
\leq
C
\|R_{h,B}\|_{H^1(B)}^2
\leq
C
\|\nabla R_{h,B}\|_{L^2(B)}^2
\nonumber\\
&\leq&
C |h|_{H^{1/2}(\partial B)}^2,\quad
\forall h\in H^{1/2}_1(\partial B).
\label{e:L2<H1/2,2}
\end{eqnarray}

To prove  \eqref{e:L2<H1/2,1} for all $\varepsilon\in(0,1)$  we would like to proceed with a scaling argument and \eqref{e:L2<H1/2,2}, which unfortunately cannot be done because the boundary $\partial B$ is not flat. 

However, it turns out that we have (see Lemma 
\ref{l:seminorms,equivalence} in Appendix for its proof):
\begin{equation}
\mbox{
$|h|_{H^{1/2}(\partial B)}$\;\, and\;\, 
${\displaystyle
[h]_{H^{1/2}(\partial B)}:=
\int_{-\pi}^\pi
d\eta
\int_{\{|\theta-\eta|<\pi\}}
\frac{|h(\theta)-h(\eta)|^2}{|\theta-\eta|^2}d\theta}$\;\,
are equivalent}, \label{e:|h|H^1/2equiv|g|H1/2}
\end{equation}
where for $x=(\cos\theta,\sin\theta)$, $y=(\cos\eta,\sin\eta)$, abusing with the notations, we have $h(\theta)=h(x)$, $h(\eta)=h(y)$.
Let us assume this result for the moment. Then \eqref{e:L2<H1/2,1} is equivalent to:
for all $h\in H^{1/2}_\varepsilon(\partial B)$ we have
\begin{equation}
 \|h\|_{L^2(\partial B)}
 \leq
 C\sqrt{{\varepsilon}}
 [h]_{H^{1/2}(\partial B)}.
 \label{e:L2<H1/2,3}
\end{equation}
From \eqref{e:L2<H1/2,2}, the estimate \eqref{e:L2<H1/2,3} holds for
$h\in H^{1/2}_1(\partial B)$.
Now, for $h\in H^{1/2}_\varepsilon\partial B)$ set 
$g(\theta) := h(\varepsilon\theta)$, $\theta\in(-\pi,\pi)$.
Then $g \in H^{1/2}_1(\partial B)$ and  satisfies \eqref{e:L2<H1/2,2}, so
\begin{equation}
\|g\|_{L^2(\partial B)} \leq C[g]_{H^{1/2}(\partial B)}.
\label{e:|g|L2<|g|H1/2}
\end{equation}
Note that by changing the variable $t = \varepsilon\theta$, $s =\varepsilon\eta$ gives
\begin{eqnarray}
[g]^2_{H^{1/2}(\partial B)}
&=&
\int_{-\pi}^{\pi}d\eta
\int_{\{|\theta-\eta|<\varepsilon\}}
\frac{|h(\varepsilon\theta) - g(\varepsilon\eta)|^2}
{|\theta-\eta|^2}
d\theta
=
\int_{-\varepsilon\pi}^{\varepsilon\pi}ds
\int_{\{|t-s|<\varepsilon\pi\}}
\frac{|h(t)-h(s)|^2}{|t-s|^2}
dt
\nonumber\\
&\leq&
[h]^2_{H^{1/2}(\partial B)}.
\label{e:|g|H1/2<|h|H1/2}
\end{eqnarray}
Note also that by changing again the variable $\varphi=\varepsilon\theta$ we have
\begin{eqnarray}
\hspace*{-5mm}
\|g\|_{L^2(\partial B)}^2
&=&
\int_{-\pi}^\pi |g(\theta)|^2d\theta
=
\int_{-\pi}^\pi |h(\varepsilon\theta)|^2d\theta
=
\varepsilon^{-1}
\int_{-\varepsilon\pi}^{\varepsilon\pi}  |h(\varphi)|^2d\varphi
=
\varepsilon^{-1}
\|h\|_{L^2(\partial B)}^2.
\label{e:|g|L2=eps|h|L2}
\end{eqnarray}
Replacing \eqref{e:|g|H1/2<|h|H1/2} and \eqref{e:|g|L2=eps|h|L2} in 
\eqref{e:|g|L2<|g|H1/2} proves the required result \eqref{e:L2<H1/2,3}.

\end{proof}
}

\paragraph*{Proof of Theorem \ref{th:e_f''(0)->convex}:} 
{As mentioned in the beginning of this Section, we start with \eqref{eq:Ef''bis} for $v$ associated to the hat function $\varphi$. From Lemma \ref{l:...phi'phi>=0} and { by assuming $x_{3}$ is small enough} we get
\begin{multline*}
\int_{\Omega_{0}}
|\nabla U'_0|^2 
+\frac{1}{2}
\int_{\partial\Om_{0}}
\!\!\!\!
f(\partial_{\nu_0} U_0)(z\cdot \nu_0)^2
v^2
ds_0+\int_{\partial\Omega_{0}}
(\partial_{\nu_{0}} U_0)^2(z\cdot\tau_0)
(z\cdot\nu_0)v\partial_{s_0} v ds_0\\
\geq \int_{\Omega_{0}}
|\nabla U'_0|^2 
-
\|f\|_{\infty}\int_{\partial\Om_0}|\partial_{\nu_{0}} U_0|V\cdot\nu_0|^2,
\end{multline*}
which in combination with Lemma \ref{l:int|DU|phi^2<}  gives {the existence of a constant $C$ such that}
\begin{multline*}
\int_{\Omega_{0}}
|\nabla U'_0|^2 
+\frac{1}{2}
\int_{\partial\Om_{0}}
\!\!\!\!
f(\partial_{\nu_0} U_0)(z\cdot \nu_0)^2
v^2
ds_0+\int_{\partial\Omega_{0}}
(\partial_{\nu_{0}} U_0)^2(z\cdot\tau_0)
(z\cdot\nu_0)v\partial_{s_0} v ds_0\\
\geq 
\int_{\Omega_{0}}|\nabla U'_0|^2 
-
C\int_{\partial\Omega_{0}}
|\nabla U_0\cdot V|^2
\nonumber
=
{| U'_0|^2_{H^{1/2}(\partial\Om_0)}} 
-
C
\|U'_0\|^2_{L^2(\partial\Om_0)},
\end{multline*}
{where we used that $U'_{0}$ is harmonic in $\Om_{0}$ (Proposition \ref{p:e_f'(0),e_f''(0)}).} 
We conclude by using Lemma \ref{l:L2<H1/2} with $w=U_{0}'$.}\qed

\paragraph*{Proof of Theorem \ref{th:l_1''(0)->convex}:}
The proof follows the one of Theorem \ref{th:e_f''(0)->convex} above, 
with 
$U_1$ instead of $U_0$ and
\eqref{e:l''(0),main} with $V=v z$ instead of \eqref{eq:Ef''bis} {(in particular Lemma \ref{l:...phi'phi>=0}, \ref{l:int|DU|phi^2<} are also valid for $U_{1}$)}.\qed

\subsection{Proof of Theorem \ref{th:main} for problem \eqref{e:Pb1}}\label{ss:proof-2.1}

{
Let $\Om_{0}$ be a solution to \eqref{e:Pb1} and 
${\Gamma_0}\subset\partial\Om_{0}\cap(D_{2}\setminus\overline{D_{1}})$ be connected.
Provided the length of $\Gamma_0$ is small enough, $\Gamma_0$ is, with respect to an $xy$ coordinate system with origin $O$ at one of its endpoints, the graph of a {positive} convex function 
$u_0\in W^{1,\infty}(0,\sigma_0)$ { such that $u_{0}(0)=0$ and $u_{0}'(0^+)=0$}.

We denote by $s_0$ the arclength variable of $\partial\Om_0$ oriented counterclockwise with origin at $O$.
We will show that $0$ is not an accumulation point of ${\rm supp}(u_0'')$ {(in other words, there exists $\eps_{0}>0$ such that $\supp(u_{0}'')\cap (0,\eps_{0})=\emptyset$)}, which geometrically means that $\Gamma_{0}$ is a locally a segment close to $O$.
By the arbitrariness of $\Gamma_0$ {(and the choice of its endpoint)} this shows that every connected component of
$\partial\Om_{0}\cap(D_{2}\setminus\overline{D_{1}})$ is polygonal.

Our proof goes by contradiction: we assume that $0$ is an accumulation point of 
${\rm supp}(u_0'')$.
}
{Then  only one of the following cases can occur:}
\begin{eqnarray}
&&
\exists \sigma\in(0,\sigma_0],
\quad
\exists (x_n)\subset(0,\sigma)\quad
\mbox{such that }\;\;
u_0''=\sum_{{n\in\N}} \alpha_n \delta(x-x_n)\;\;
\mbox{in $(0,\sigma)$},
\quad
\label{e:B1}\\
\mbox{or}&&
\exists (x_n)\subset(0,\sigma_0)\quad
\mbox{such that every $x_n$ is an accumulation point of }{\rm supp}(u_0'').
\label{e:B2}
\end{eqnarray}

Indeed, if \eqref{e:B1} does not hold, necessarily there exists an accumulation point 
$x_1\in(0,\sigma_0)\cap {\rm supp}(u_0'')$ {(otherwise} \eqref{e:B1} would hold in $(0,\sigma_0)$). 
Repeating the same argument in $(0,x_n)$ provides $x_{n+1}$, an accumulation point of
${\rm supp}(u_0'')$ with $x_{n+1}<x_n$, for all $n\in\mathbb N$.
Without loss of generality we may assume $(x_n)$ strictly decreasing in both cases.

In the following we proceed in two steps:
\begin{itemize}
\item First we show that Theorem \ref{th:e_f''(0)->convex} applies if $n$ is large enough. To that end, we mainly need to show that {assumption $(ii)$ in Theorem \ref{th:e_f''(0)->convex}} is satisfied, which relies on a first order optimality condition for \eqref{e:Pb1}. Theorem \ref{th:e_f''(0)->convex} then shows that {$E_{f}''(\Om_{0})(V_n,V_n)>0$} for $n$ large enough, where $V_n(x,u_0(x))=(0,\varphi_n(x))$ and
$\varphi_n\in H^1_0(0,\sigma)$ is the hat function associated to the nodes 
$\{x_{n-1}, x_n,x_{n+1}\}$ (see Definition \ref{d:hat}).
\item The second step is to use the second order optimality condition to prove that on the contrary {$E_{f}''(\Om_{0})(V_n,V_n)\leq 0$}. 
\end{itemize}
We emphasize that in the case of 
\eqref{e:B2}, {$V_{n}$ is not an admissible deformation for $\Om_0$ (equivalently, $\varphi_n$ does not preserve the convexity of $u_0$, see Definition \ref{def:convexgraphe})}, which is an obstacle for both steps. Nevertheless thanks to the results in Section \ref{sect:perturbation}, $V_n$ and $\varphi_n$ are close in a suitable sense to admissible perturbations.

\begin{lemma}\label{l:varphi}
Let $\Om_{0}$ be a solution to \eqref{e:Pb1}, $\Gamma_{0}=\{(x,u_{0}(x)), x\in (0,\sigma_{0})\}$ as above, satisfying \eqref{e:B1} or \eqref{e:B2}. 
Then  for all $n$ {and all $\varphi$ {affine} on $[x_{n+1},x_{n}]$}, we have
\begin{equation}\label{eq:EL0}
\int_{x_{n+1}}^{x_n}
\left(
\mu-
\frac{1}{2}
(\partial_{\nu_{0}} U_0)^2
\right)
\varphi dx
=0.
\end{equation}
\end{lemma}
\begin{proof}
The proofs {differ whether we assume} \eqref{e:B1} or \eqref{e:B2}.
{In each case  we consider $\varphi_{n}$, $n\geq 1$, the hat function associated to
$\{x_{n+1},x_{n},x_{n-1}\}$ and set 
$v_n(x,u_0(x))=\varphi_n(x)$,
{$V_{n}=\bar{v}_{n}z$ where $\bar{v}_{n}$ is a Lipschitz extension of $v_{n}$}. For the purpose of this lemma  $z\in\R^2$ will be choosen appropriately for each of cases \eqref{e:B1}, \eqref{e:B2}.
}

Assume first that \eqref{e:B1} holds. 
{So {$\Gamma_{0}$ restricted to $(0,\sigma)$} is a union of segments {accumulating to $O$}, 
{whose corners are $P_{n}=(x_{n},u_{0}(x_{n}))$}.} 
We take {$z\parallel(P_n-P_{n-1})$, so that $z\cdot\nu_0=0$ in $(P_n,P_{n-1})$}.
Then $(I+tV_n)(\Omega_{0})$ is an admissible domain for all $|t|$ small.
So the first order optimality condition holds, which by using  \eqref{e:e_f'}
gives 
\[
0
=
\int_{\partial\Omega_0}
\left(
\mu-
\frac{1}{2}
(\partial_{\nu_{0}} U_0)^2
\right)
(z\cdot\nu_0)v_n ds_0
=
\int_{P_{n+1}}^{P_n}
\left(
\mu-
\frac{1}{2}
(\partial_{\nu_{0}} U_0)^2
\right)
(z\cdot\nu_0)v_n ds_0.
\]
Changing the variable $ds_0=\sqrt{1+(u_0'(x))^2}dx$ {with $\sqrt{1+(u_0')^2}$ being constant in $(x_{n+1},x_n)$}, 
and  taking into account the fact that $z\cdot\nu_0$ is also constant in $(P_{n+1},P_n)$,
the equality above gives
\begin{equation}
0
=
\int_{x_{n+1}}^{x_n}
\left(
\mu-
\frac{1}{2}
(\partial_{\nu_{0}} U_0)^2
\right)
{\varphi_{n}}
dx.
\label{e:EL2}
\end{equation}
{Reproducing the same argument with the hat function $\varphi_{n+1}$ associated to 
$\{P_{n+2},P_{n+1},P_{n}\}$ with $z\parallel(P_{n+2}-P_{n+1})$ we obtain the same result \eqref{e:EL2} with 
$\varphi_{n+1}$ instead of $\varphi_{n}$. As $\{\varphi_{n},\varphi_{n+1}\}$ is a basis of affine functions on $[x_{n},x_{n+1}]$, this proves \eqref{eq:EL0} in the case \eqref{e:B1}.}

Now assume \eqref{e:B2}.
{In this case we take $z=(0,1)$ and  consider
{$(\varphi^{-}_{k,n})_k$, the sequence of functions given by $(ii)$} in Proposition \ref{p:varphi} 
associated to $\{x_{n+1},x_n,x_{n-1}\}$. As usual we then set
$v_{k,n}(x,u_0(x))={\varphi^{-}_{k,n}}(x)$,
$V_{k,n}=\bar{v}_{k,n}z$.
}
{Then 
$(I+tV_{k,n})(\Omega_0)$, is an admissible domain for all $|t|$ small.}
Using the first order optimality condition leads to
\begin{eqnarray*}
0
&=&
\int_{\partial\Omega_0}
\left(
\mu-
\frac{1}{2}
(\partial_{\nu_{0}} U_0)^2
\right)
(V_{k,n}\cdot\nu_0)ds_0
=
\int_0^\sigma
\left(
\mu-
\frac{1}{2}
(\partial_{\nu_{0}} U_0)^2
\right)
(u_0',-1)\cdot(z_1,z_2){\varphi^{-}_{k,n}} dx
\\
&=&
\int_0^\sigma
\left(
\mu-
\frac{1}{2}
(\partial_{\nu_{0}} U_0)^2
\right)
{\varphi^{-}_{k,n}} dx.
\end{eqnarray*}
{From  the $L^1(0,\sigma)$ convergence of {$(\varphi^-_{k,n})$} to 
{$\varphi_{n}\1_{(x_{n+1},x_n)}$}, we can take the limit} $k\to\infty$ in the previous equality and obtain 
\eqref{eq:EL0} for {$\varphi=\varphi_n$}. 
{Finally we proceed as above with $(\varphi^+_{n+1,k})_k$ the sequence of functions obtained with 
$(ii)$ of Proposition \ref{p:varphi} associated to $\{x_{n+2},x_{n+1},x_{n}\}$, {converging to $\varphi_{n+1}\1_{(x_{n+1},x_n)}$}. This proves 
\eqref{eq:EL0} with $\varphi=\varphi_{n+1}$ and completes the proof of the lemma, 
{using again that $\{\varphi_{n},\varphi_{n+1}\}$ is a basis of affine functions on $[x_{n},x_{n+1}]$}.}
\end{proof}

{
\paragraph*{End of proof of Theorem \ref{th:main} for $\Om_{0}$ solution to \eqref{e:Pb1}.}
{Let $\varphi_n$ be the hat function associated to the nodes $\{x_{n+1},x_n,x_{n-1}\}$ and set
 $v_n(x,u_0(x))= \varphi_n(x)$, $V_n={\overline{v}_n} z$, $z=(0,1)$.}
\begin{itemize}
\item
From \eqref{e:m''(0)bis} we have ${{\rm Vol}''}(\Om_0)(V_n,V_n)=0$.
\item 
From Lemma \ref{l:varphi}, one can apply Theorem \ref{th:e_f''(0)->convex} which shows 
{\begin{multline}
\forall n\gg 1, \;\;\int_{\Omega_{0}}
|\nabla U'_{0,n}|^2 
+\frac{1}{2}
\int_{\partial\Om_{0}}
\!\!\!\!
f(\partial_{\nu_0} U_0)(z\cdot \nu_0)^2
v_{n}^2
ds_0\nonumber\\+\int_{\partial\Omega_{0}}
(\partial_{\nu_{0}} U_0)^2(z\cdot\tau_0)
(z\cdot\nu_0)v_{n}\partial_{s_0} v_{n} ds_0>0.
\label{e:Ef''>0}
\end{multline}
where $U_{0,n}'$ solves \eqref{e:U'} with $V=V_{n}$.}

\item On the other hand, let us show that the second order optimality condition for \eqref{e:Pb1} implies
\begin{multline*}
\forall n\gg 1, \;\;\int_{\Omega_{0}}
|\nabla U'_{0,n}|^2 
+\frac{1}{2}
\int_{\partial\Om_{0}}
\!\!\!\!
f(\partial_{\nu_0} U_0)(z\cdot \nu_0)^2
v_{n}^2
ds_0\nonumber\\+\int_{\partial\Omega_{0}}
(\partial_{\nu_{0}} U_0)^2(z\cdot\tau_0)
(z\cdot\nu_0)v_{n}\partial_{s_0} v_{n} ds_0\leq 0.
\end{multline*}
Again we consider two cases:  if \eqref{e:B1} holds then $V_{n}$ are admissible deformations, and the result follows from $E_{f}''(\Om)(V_{n},V_{n})\leq 0$ {and Theorem \ref{th:e_f''(0)=}}.
If \eqref{e:B2} holds, let $(\varphi_{k,n})$ be the sequence given by 
{$(i)$} in Proposition \ref{p:varphi} associated to nodes $\{x_{n+1},x_n, x_{n-1}\}$.
Then $V_{k,n}(x,u_0(x))=v_{k,n}(x) z$ is an admissible deformation and implies
\begin{equation}\label{eq:ELL}
\forall n\gg 1, \forall k\gg1,  \;\;E_{f}''(\Om)(V_{k,n},V_{k,n})\leq 0.
\end{equation}
which leads with Theorem \ref{th:e_f''(0)=} to 
{\begin{multline}\label{eq:ELL2}
\forall n\gg 1, \forall k\gg1,\;\;\int_{\Omega_{0}}
|\nabla U'_{0,k,n}|^2 
+\frac{1}{2}
\int_{\partial\Om_{0}}
\!\!\!\!
f(\partial_{\nu_0} U_0)(z\cdot \nu_0)^2
v_{k,n}^2
ds_0\nonumber\\+\int_{\partial\Omega_{0}}
(\partial_{\nu_{0}} U_0)^2(z\cdot\tau_0)
(z\cdot\nu_0)v_{k,n}\partial_{s_0} v_{k,n} ds_0\leq 0.
\end{multline}}
By $H^1$ convergence of $(\varphi_{k,n})$ to $\varphi_{n}$ when $k$ goes to infinity, we can take the limit $k\to\infty$ in the previous inequality.
\end{itemize}
Of course, this constitutes a contradiction and concludes the proof of Theorem \ref{th:main} for problem \eqref{e:Pb1}.\qed

{\subsection{Proof of Theorem \ref{th:main} for problem \eqref{e:Pb2}}
First, we note that Lemma 
\ref{l:varphi} hold with $U_1$ instead of $U_0$.
Then we follow the proof of Theorem \ref{th:main} for problem \eqref{e:Pb1},
using Theorem \ref{th:l_1''(0)->convex} instead of Theorem \ref{th:e_f''(0)->convex}.
\hfill$\Box$\\
}

\section{Proof of Theorem \ref{th:main} for the problem \eqref{e:Pb3} and \eqref{e:Pb4}}\label{sect:other}
{The proof {of Theorem \ref{th:main} for} \eqref{e:Pb3} and \eqref{e:Pb4}  is similar to the proof  for problems \eqref{e:Pb1} and \eqref{e:Pb2}.}
The particularity is that we need to find perturbations which 
preserve {area} and convexity, and satisfy the preliminary results of Section \ref{ss:convexity-e''}. {As in Section \ref{ss:convexity-e''}, we focus mainly on \eqref{e:Pb3} involving $E_{f}$, but everything can easily be adapted to \eqref{e:Pb4} involving $\lambda_{1}$.}

{In the whole section, $\Om_0$ is a solution of \eqref{e:Pb3}, $\Gamma_{0}\subset \partial\Om_{0}$, and $u_{0}$ is such that
{$\Gamma_{0}=\{(x,u_{0}(x)), x\in (0,\sigma_{0})\}$ with $u_{0}(0)=u_{0}'(0)=0$. Also, we assume $u_{0}$ satisfies \eqref{e:B1} or \eqref{e:B2}, as described in the beginning of Section \ref{ss:proof-2.1}.} 
For every $n$, $\varphi_n$ is the hat function associated to the nodes $\{x_{n+1},x_n,x_{n-1}\}$,
$v_n(x,u_0(x))=\varphi_n(x)$ for $x\in(0,\sigma_0)$ and $v_n=0$ elsewhere, and {$\bar{v}_{n}$ is the $W^{1,\infty}(\R^2)$-extension of $v_n$ (as for example given by Lemma \ref{l:Hv})}.
}

We start with the following result replacing Lemma \ref{l:varphi} and proving that we have a first order optimality condition under  {area} and convexity constraint 
\begin{lemma}\label{l:varphi,m}
There exists $\mu>0$ such that \eqref{eq:EL0} holds {for every $n\in\N$ and $\varphi$ affine in $[x_{n+1},x_{n}]$}.
\end{lemma}
{\bf Proof}.
{We note that \eqref{eq:EL0} is equivalent to 
{\begin{equation}
\frac{1}{|x_n-x_{{n-1}}|}\int_{x_n}^{x_{n-1}} |\nabla U_0|^2 \varphi_{n} dx=\frac{1}{|x_n-x_{{n-1}}|}\int_{x_n}^{x_{n-1}} |\nabla U_0|^2 \varphi_{n-1} dx= \mu.
\label{eq:EL0+}
\end{equation}}}
Assume first that \eqref{e:B1} holds. 
{Denoting $P_{n}=(x_{n},u_{0}(x_{n}))$, we also introduce} $\tau_n^+=(P_{n+1}-P_n)/|P_{n+1}-P_n|$,
$\tau_n^-=(P_{n-1}-P_n)/|P_{n-1}-P_n|$, {and finally}
$W_n(x,t)=\alpha_n(t){\bar{v}_n}(x)\tau_n^+ + \alpha_{n-1}(t){\bar{v}_{n-1}}(x)\tau_{n-1}^-$ and $\Omega_t=(I+W_n{(\cdot,t)})(\Omega_0)$. 
We choose $\alpha_n$ and $\alpha_{n-1}$ differentiable so that {$\alpha_{n}(0)=0$}, 
$m(t)=|\Omega_t|=m_0$ { and $\alpha_{n}'(0)\neq 0$}, for example, we can take $\alpha_n(t)=t$ and then solve $\alpha_{n-1}(t)$ explicitly. 
It follows
\begin{eqnarray*}
m'(0)&=&
 \int_{P_n}^{P_{n-1}}(\partial_tW_n(\cdot,0)\cdot\nu_0)ds_0
\nonumber \\
&=&
 \int_{P_n}^{P_{n-1}}
 (\alpha_n'(0)v_n(\tau_n^+\cdot\nu_0)
 +
 \alpha_{n-1}'(0)v_{n-1}(\tau_{n-1}^-\cdot\nu_0))ds_0
 =
 0,
\end{eqnarray*}
which, as  
${\displaystyle 
\int_{P_n}^{P_{n-1}}v_n ds = \int_{P_n}^{P_{n-1}}v_{n-1} ds}$,  implies
\begin{equation}
\alpha_n'(0)(\tau_n^+\cdot\nu_0) 
 +
 \alpha_{n-1}'(0)(\tau_{n-1}^-\cdot\nu_0)
 =
 0,
\label{e:m'(0),1}
\end{equation}
{where $\nu_{0}$ is the normal to $\partial\Om_{0}$ and so it is constant on $(P_{n},P_{n-1})$}.

Using the optimality condition for $t\mapsto E_f(\Om_t)$ gives
\begin{eqnarray*}
\frac{d}{dt}E_f(\Om_t)\Big|_{t=0}
&=&
E_f'(\Om_0)({\partial_{t}} W_n(\cdot,0))
=
{-\frac12}\int_{\partial\Om_0}|\nabla U_0|^2({\partial_{t}} W_n(\cdot,0)\cdot\nu_0)dx
\\
&=&
{-\frac12}\left(\int_{P_{n}}^{P_{n-1}}
|\nabla U_0|^2
\left(
\alpha_n'(0) (\tau_n^+\cdot\nu_0)v_n 
+
 \alpha_{n-1}'(0)(\tau_{n-1}^-\cdot\nu_0)v_{n-1}
 \right)ds_0\right)
 =
 0,
\end{eqnarray*}
which combined with \eqref{e:m'(0),1} and after changing the variable $ds_0=\sqrt{1+|u_0'|^2}dx$ gives 
\begin{eqnarray}
\int_{x_{n}}^{x_{n-1}}
|\nabla U_0|^2\varphi_ndx
&=&
\int_{x_{n}}^{x_{n-1}}
|\nabla U_0|^2\varphi_{n-1}dx,
\label{e:E'=mu,1}
\end{eqnarray}
because $\sqrt{1+|u_0'|^2}$ is constant on $[x_n,x_{n-1}]$ { and $\alpha_n'(0) (\tau_n^+\cdot\nu_0)\neq 0$}.

{We reproduce the same computation with } 
$W_n(x,t)
=
\alpha_{n+1}(t){\bar{v}}_{n+1}(x)\tau_{n+1}^+ 
+ 
\alpha_{n-1}(t){\bar{v}}_{n-1}(x)\tau_{n-1}^-$ such that 
$|\Omega_t|=m_0$. It follows
\begin{eqnarray}
m'(0)=
 \int_{P_{n+1}}^{P_n}
 \alpha_{n+1}'(0)(\tau_{n+1}^+\cdot\nu_0) v_{n+1}
 +
 \int_{P_n}^{P_{n-1}}
 \alpha_{n-1}'(0)(\tau_{n-1}^-\cdot\nu_0)v_{n-1}
 &=&0,
 \end{eqnarray}
 and as
${\displaystyle 
\int_{P_{n+1}}^{P_n} v_n ds =\int_{P_{n+1}}^{P_n} v_{n+1} ds = \frac12|P_{n+1}-P_n|}$, we get
\begin{equation} 
 \alpha_{n+1}'(0)(\tau_{n+1}^+\cdot\nu_0)|P_{n+1}-P_n| 
 +
 \alpha_{n-1}'(0)(\tau_{n-1}^-\cdot\nu_0)|P_n-P_{n-1}|
=
 0,
 \label{e:m'(0),2}
\end{equation}
where again in $(\tau_{n+1}^+\cdot\nu_0)$,  $\nu_{0}$ refers to the normal to $\partial\Om_{0}$ restricted to $(P_{n+1},P_{n})$ which is constant).
Using again  the optimality condition for $t\mapsto E_f(\Om_t)$ and \eqref{e:m'(0),2} gives
\begin{eqnarray*}
0
&=&
\int_{P_{n+1}}^{P_n}|\nabla U_0|^2({\partial_{t}} W_n(\cdot,0)\cdot\nu_0)ds_0
+
\int_{P_n}^{P_{n-1}}|\nabla U_0|^2({\partial_{t}} W_n(\cdot,0)\cdot\nu_0)ds_0
\\
&=&
\int_{P_{n+1}}^{P_n}
|\nabla U_0|^2
\alpha_{n+1}'(0)(\tau_{n+1}^+\cdot\nu_0)v_{n+1}ds_0
+
\int_{P_{n}}^{P_{n-1}}
|\nabla U_0|^2
\alpha_{n-1}'(0)(\tau_{n-1}^-\cdot\nu_0)v_{n-1}ds_0
\\
&=&
\frac{1}{|P_{n+1}-P_n|}
\int_{P_{n+1}}^{P_n}
|\nabla U_0|^2
v_{n+1}ds_0
-
\frac{1}{|P_{n+1}-P_n|}
\int_{P_{n}}^{P_{n-1}}
|\nabla U_0|^2
v_{n-1}ds_0,
\end{eqnarray*}
which after  changing the variable 
 and noticing that
${\displaystyle
\frac{|P_{n+1}-P_n|}{\sqrt{1+|u_0'|^2}}=|x_{n+1}-x_n|}$ gives
\begin{eqnarray}
\frac{1}{|x_{n+1}-x_n|}
\int_{x_{n+1}}^{x_{n}}
|\nabla U_0|^2 \varphi_{n+1}
dx
&=&
\frac{1}{|x_n-x_{n-1}|}
\int_{x_{n}}^{x_{n-1}}
|\nabla U_0|^2 \varphi_{n-1}
dx.
\label{e:E'=mu,2}
\end{eqnarray}
Combining \eqref{e:E'=mu,1} with \eqref{e:E'=mu,2} {allows to define $\mu:=\displaystyle{\frac{1}{|x_{n}-x_{n-1}|}\int_{x_{n}}^{x_{n-1}}
|\nabla U_0|^2 \varphi_{n}
dx}$ (independent on $n$), which proves \eqref{eq:EL0+} in the case \eqref{e:B1}.}

In the case of the assumption \eqref{e:B2} we use Proposition  \ref{p:varphi}.
For every $n$ we consider the sequence {$(\varphi_{k,n}^{-})$}, resp. {$(\varphi_{k,n}^{+})$}, associated to 
$\{x_{n+1},x_n,x_{n-1}\}$, as given by $(ii)$ of  Proposition \ref{p:varphi}, which converges to 
$\varphi_n\1_{(x_{n+1},x_n)}$, resp. $\varphi_n\1_{(x_n,x_{n-1})}$, in $L^1(0,\sigma_0)$.
We set $v_{k,n}^{-}(x,u_0(x))=\varphi_{k,n}^{-}(x)$, $v_{k,n}^{+}(x,u_0(x))=\varphi_{k,n}^{+}(x)$ for $x\in(0,\sigma_0)$ 
and extended by zero elsewhere. Furthermore, let
{$\bar{v}_{k,n}^{-}$}, resp. {$\bar{v}_{k,n}^{+}$}, be a {$W^{1,\infty}(\R^2)$} extension of $v_{k,n}^{-}$, 
resp. $v_{k,n}^{+}$.

Now, let $\Om_t=(I+{W_n(\cdot,t)})(\Om_0)$ with 
$W_n(x,t)=(\alpha_n(t){\bar{v}_{k,n}^{-}(x)}+\alpha_{n-1}(t){\bar{v}_{k,n-1}^{-}(x)})z$, {where} $z=(0,1)$.
Proposition \ref{p:varphi} ensures that $\Om_t$ are convex for all $|t|$ small.
As before, we choose $\alpha_n$ and $\alpha_{n-1}$ such the area of $\Om_t$ is $m_0$, {i.e. satisfying}
$\alpha_n(t)
\int_0^\sigma  \varphi_{k,n}^{+} dx
+
\alpha_{n-1}(t)
\int_0^\sigma  \varphi_{k,n-1}^{-} dx
=
0$.
By differentiating at $t=0$ it gives
\begin{eqnarray}
\alpha_n'(0)
\int_0^\sigma  \varphi_{k,n}^{+} 
+
\alpha_{n-1}'(0)
\int_0^\sigma  \varphi_{k,n-1}^{-} 
&=&
0,
\end{eqnarray}
which in limit $k\to\infty$ gives
\begin{eqnarray}\alpha_n'(0)
\int_{x_n}^{x_{n-1}}  \varphi_n^{+} dx
+
\alpha_{n-1}'(0)
\int_{x_n}^{x_{n-1}}  \varphi_{n-1}^{-} dx
&=&
0,\quad
\mbox{i.e. }\;\;\;
\alpha_n'(0)+\alpha_{n-1}'(0)=0.
\label{e:an'+an-1'=0,1}
\end{eqnarray}
By using the optimality condition for $E_f(\Om_t)$
we get
\begin{eqnarray}
0
&=&
\int_{\partial\Om_0}|\nabla U_0|^2(\partial_t W_n(\cdot,0)\cdot\nu_0)ds_0
\nonumber\\
&=&
\int_{\partial\Om_0}
|\nabla U_0|^2(\alpha_n'(0)v_{k,n}^{+}(z\cdot\nu_0) + \alpha_{n-1}'(0)v_{k,n-1}^{-}(z\cdot\nu_0))ds_0
\nonumber\\
&=&
\int_0^\sigma 
|\nabla U_0|^2
\left(
\alpha_n'(0)
\varphi_{k,n}^{+} 
+
\alpha_{n-1}'(0)
\varphi_{k,n-1}^{-}
\right)
dx.
\label{e:int=int,0}
\end{eqnarray}
Letting $k\to\infty$ in \eqref{e:int=int,0} and using  \eqref{e:an'+an-1'=0,1} gives
\begin{equation}
\int_{x_n}^{x_{n-1}}
|\nabla U_0|^2
\varphi_n
dx
=
\int_{x_n}^{x_{n-1}}
|\nabla U_0|^2
\varphi_{n-1}
dx
\label{e:int=int}
\end{equation}
To complete the proof  we finally consider
$\Om_t=(I+W_n{(\cdot,t)})(\Om_0)$ with $W_n=(\alpha_{n}^-(t)\bar{v}_{k,n}^{-}+\alpha_{n}^+(t)\bar{v}_{k,n}^{+})z$, 
such that the area of $\Om_t$ is $m_0$. Proceeding as above it implies
\begin{eqnarray*}
(\alpha^-_{n})'(0)
\int_0^\sigma  \varphi_{k,n}^{-} dx
+
(\alpha_{n}^+)'(0)
\int_0^\sigma  \varphi_{k,n}^{+} dx
&=&
0,
\end{eqnarray*}
{which in limit gives}
\begin{eqnarray}(\alpha_{n}^-)'(0)|x_{n+1}-x_{n}|+(\alpha_{n}^+)'(0)|x_{n}-x_{n-1}|
&=&
0.
\label{e:an'+an-1'=0,2}
\end{eqnarray}
Writing the optimality of $t\mapsto E_f(\Om_t)$ gives
\begin{eqnarray*}
0
&=&
\int_{\partial\Om_0}|\nabla U_0|^2(\partial_t W_n(\cdot,0)\cdot\nu_0)ds_0
\\
&=&
\int_{\partial\Om_0}|\nabla U_0|^2
\left(
(\alpha_{n}^-)'(0)
v_{k,n}^{-}
+
(\alpha_{n}^+)'(0)
v_{k,n}^{+} 
\right)
(z\cdot\nu_0)
ds_0
\\
&=&
\int_0^\sigma
|\nabla U_0|^2
\left(
(\alpha_{n}^-)'(0)
\varphi_{k,n}^{-}
+
(\alpha_{n}^+)'(0)
\varphi_{k,n}^{+} 
\right)
dx,
\end{eqnarray*}
which after letting $k\to\infty$
 gives
\[
\frac{1}{|x_{n+1}-x_n|}
\int_{x_{n+1}}^{x_n}
|\nabla U_0|^2
\varphi_n
dx
=
\frac{1}{|x_{n}-x_{n-1}|}
\int_{x_n}^{x_{n-1}}
|\nabla U_0|^2
{\varphi_{n}}
dx
.
\]
{Similarly to the previous case, this last equality} combined with \eqref{e:int=int}  completes the proof of lemma 
for the problem \eqref{e:Pb3}.

\hfill$\Box$

{{The following two lemmas will replace Lemma \ref{l:...phi'phi>=0} and \ref{l:int|DU|phi^2<}:}
\begin{lemma}\label{l:...phi'phi>=0,m}
{For every} $(n, m)$ with  $m>n+1$
, {there exists 
$(\alpha_n, {\beta}_m)\in\R^2\setminus\{(0,0)\}$} such that if
${w}_{n,m}=\alpha_{n}\bar{v}_{n}+{\beta}_{m}\bar{v}_{m}$, 
$W_{n,m}=(0,{{w}_{n,m}})$,
then for $|t|$ small 
$\Omega_t=(I+tW_{n,m})(\Omega_0)$ is convex { and } $|\Omega_t|=m_0$.
{Moreover, if $n$ is large enough, then} 
\[
\int_{\partial\Omega_0}
(\partial_{\nu_{0}} U_0)^2(z\cdot\tau_0)(z\cdot\nu_0)w_{n,m}\partial_{s_0} w_{n,m} ds_0
\geq0.
\]
\end{lemma}
{\bf Proof}.
{Let $(n,m)$ such that $n>m+1$ so that $\varphi_{n}$ and $\varphi_{m}$ have disjoint support.} We choose $\alpha_n=\frac{1}{|x_{n-1}-x_{n+1}|}$, ${\beta}_m=-\frac{1}{|x_{m-1}-x_{m+1}|}$ so that
\begin{equation*}\label{e:psi}
\alpha_n\int_{x_{n+1}}^{x_{n-1}}\varphi_n 
 + 
 {\beta}_m\int_{x_{m-1}}^{x_{m+1}}\varphi_m
 =0
\end{equation*}
{which leads to $|\Omega_t|=m_0$. }
As $v_n$ and $v_m$ have disjoint supports, {we have}
\begin{eqnarray*}
\int_{\partial\Omega_0}
(\partial_{\nu_{0}} U_0)^2(z\cdot\tau_0)(z\cdot\nu_0) w_{n,m}\partial_{s_0} w_{n,m} ds_0
&=&
\alpha_n^2
\int_{\partial\Omega_0}
(\partial_{\nu_{0}} U_0)^2(z\cdot\tau_0)(z\cdot\nu_0) v_n\partial_{s_0} v_n ds_0
\\&&+
{\beta}_m^2
\int_{\partial\Omega_0}
(\partial_{\nu_{0}} U_0)^2(z\cdot\tau_0)(z\cdot\nu_0) v_m\partial_{s_0} v_m ds_0
\end{eqnarray*} 
{which is nonnegative if $n$ is large enough, by Lemma \ref{l:...phi'phi>=0}.}
}
\hfill$\Box$
}

{
\begin{lemma}\label{l:int|DU|phi^2<,m}
There exists $C$ such that the function $w_{n,m}$ of Lemma \ref{l:...phi'phi>=0,m} satisfies
\[
\int_{\partial \Om_{0}}|\partial_{\nu_{0}}U_0| w_{n,m}^2ds_0
\leq 
C
\int_{\partial \Om_{0}}|\partial_{\nu_{0}}U_0|^2 w_{n,m}^2 ds_0
\]
{for all $n$ large enough (and $m>n+1$)}
\end{lemma}
\begin{proof}
{As $\bar{v}_n$ and $\bar{v}_m$ have disjoint supports we have $w_{n,m}^2=\alpha_n^2 \bar{v}_n^2 + {\beta}_m^2 \bar{v}_m^2$. Thanks to Lemma \ref{l:varphi,m}, we satisfy the conditions of Theorem \ref{th:e_f''(0)->convex} for every triplet $\{x_{n+1},x_{n},x_{n-1}\}$, which means we can apply Lemma \ref{l:int|DU|phi^2<} to $v_{n}$ and $v_{m}$, which leads to the result.}
\end{proof}}

{
\paragraph*{Proof of Theorem \ref{th:main} for problems \eqref{e:Pb3} and \eqref{e:Pb4}:}
We procced as in the proofs of Theorem \ref{th:main} for problems \eqref{e:Pb1} and
\eqref{e:Pb2}, with the difference that we use the functions $\psi_{n,m}$, $w_{n,m}$, $W_{n,m}=(0,{w}_{n,m})$ 
(and their versions with the index $k$ for the case \eqref{e:B2}) and Lemmas 
\ref{l:varphi,m}, \ref{l:...phi'phi>=0,m}, \ref{l:int|DU|phi^2<,m} instead of 
 $\varphi_n$, $v_n$, $V_n=(0,\bar{v}_n)$ (and their version with index $k$), 
and Lemmas \ref{l:...phi'phi>=0}, \ref{l:int|DU|phi^2<},  \ref{l:varphi}.
\hfill$\Box$
}

\section{Appendix}
{The following proposition gives a complete expression for $E_f''(\Om)(V,V)$ in terms of boundary integrals 
in cases when $\Omega$ is smooth or convex, and 
extends the definition of ${\cal H}|\nabla U|^2$ as an element of $(W^{1,\infty}(\partial\Om))'$ in the case $\Om$ is convex.
\begin{proposition}\label{p:E_f''-complete}
Under the  assumptions of Theorem \ref{th:e_f''(0)=}, except for $\Omega$, the formula \eqref{eq:Ef''} holds if:
\\
i) $\Omega$ is of class $C^{2,\sigma}$, for any $\sigma\in(0,1)$, where the term 
$\int_{\partial\Om}{\cal H} |\nabla U|^2 |V|^2$ is understood with ${\cal H}$ the curvature in the clasical sense, or
\\
ii) $\Omega$ is convex, and in this case the term with ${\cal H}$ is understood as
\begin{equation}
\int_{\partial\Om}
{\cal H} |\nabla U|^2 |V|^2
:=
-
\int_{\Om}
\left[
\partial_j\partial_i^\perp U
\right]
\partial_i
\left(
\partial_j^\perp U |V|^2
\right)
=
\lim_{\eps\to0}\int_{\partial\Om_\eps}{\cal H}_\eps|\nabla U|^2 |V|^2 ds_\eps.
\label{e:def-H|DU|^2}
\end{equation}
Similarly, under the assumptions of Theorem \ref{th:l_1''(0)=}, the same statements as above hold for $\lambda_1''(\Om)(V,V)$ 
given by
\begin{eqnarray}
\hspace*{-11mm}
\lambda_1''(\Om)(V,V)
&=&
\!\!\int_{\Om}
|\nabla U_{1}'|^2 
+
\int_{\partial\Om}
\Big(
{\frac{1}{2}\lambda_{1}(\Om) ({\partial_{\nu} |U_{1}|^2)}(V\cdot\nu)^2}
\nonumber\\
&&
\hspace*{17mm}
+
\int_{\partial\Om}
\frac{1}{2}
{\cal H}(\partial_{\nu}U_{1})^2|V|^2
+
(\partial_{\nu}U_{1})^2 (V\cdot\tau)(\partial_{s} V\cdot \nu)
\Big)ds.
\end{eqnarray}
\end{proposition}
{\bf Proof}.
First we prove the result for $E_f''(\Om)$.
In both i) and ii) cases we  follow the proof of Theorem \ref{th:e_f''(0)=}. 
In the case i), we can repeat all the calculus of Theorem \ref{th:e_f''(0)=} with $\Omega_\varepsilon$ replaced by 
$\Omega$ (and so, without the need to consider the limits of different terms as $\varepsilon$ tends to zero).

In the case ii), we need to identify  $\lim_{\varepsilon\to0}K_3(\varepsilon)$.}
For this we use the following extensions for the normal and tangential vectors on $\partial\Om_\eps$, 
$\nu_\eps=-\frac{\nabla U}{|\nabla U|}$, 
$\tau_\eps=\nu_\eps^\perp=-\frac{\nabla^\perp U}{|\nabla U|}$.
Then we have
\begin{eqnarray}
K_3(\eps)
&=&
\int_{\partial\Om_\eps}
{\cal H}_\eps |\nabla U|^2 |V|^2
\nonumber\\
&=&
-
\int_{\partial\Om_\eps}
|\nabla U|^2(\nu_\eps\cdot\partial_{s_\eps}\tau)|V|^2
\nonumber\\
&=&
-
\int_{\partial\Om_\eps}
|\nabla U|^2
\left(
\nu_\eps
\cdot
\nabla\left[\frac{\nabla^\perp U}{|\nabla U|}\right]
\cdot
\frac{\nabla^\perp U}{|\nabla U|}\right)
|V|^2
\nonumber\\
&=&
-
\int_{\partial\Om_\eps}
|\nabla U|^2
\left(
\nu_\eps^i
\left[
\frac{\partial_j\partial_i^\perp U}{|\nabla U|} 
- 
\frac{\partial_i^\perp U\partial_{jk}U\partial_kU}{|\nabla U|^3}
\right]
\frac{\partial_j^\perp U}{|\nabla U|}
\right)
|V|^2
\nonumber\\
&=&
-
\int_{\partial\Om_\eps}
\nu_\eps^i
\left(
\left[
\partial_j\partial_i^\perp U
\right]
\partial_j^\perp U
|V|^2
\right)
+
\int_{\partial\Om_\eps}
(\nu_i\partial_i^\perp U)
\left[
\frac{\partial_{jk}U\partial_kU\partial_j^\perp U}{|\nabla U|^2}
\right]
|V|^2
\nonumber\\
&=&
-
\int_{\Om_\eps}
\partial_i
\left(
\left[
\partial_j\partial_i^\perp U
\right]
\partial_j^\perp U
|V|^2
\right)
\nonumber\\
&=&
-
\int_{\Om_\eps}
\left[
\partial_j\partial_i^\perp U
\right]
\partial_i
\left(
\partial_j^\perp U |V|^2
\right)
\nonumber\\
&\xrightarrow[\eps\to0]{}&
-
\int_{\Om}
\left[
\partial_j\partial_i^\perp U
\right]
\partial_i
\left(
\partial_j^\perp U |V|^2
\right)
\nonumber\\
&=:&
\int_{\partial\Om}
{\cal H} |\nabla U|^2 |V|^2.
\label{e:K3}
\end{eqnarray}
The proof of the result for $\lambda_1''(\Om)$ is similar to the proof for $E_f''(\Om)$ and we do not present it here.
\hfill$\Box$

{
\begin{remark}
In the case $\Om$ convex, ${\cal H}|\nabla U|^2$ is not well defined because, for example, 
${\cal H}$ is unbounded and $|\nabla U|$ is zero at a corner.
However, as the calculus proceeding \eqref{e:K3} holds
with $|V|^2=\bar{w}$, for every $w\in W^{1,\infty}(\partial\Om)$ and $\bar{w}$ the extension of $w$ given by Lemma \ref{l:Hv}, it implies that \eqref{e:def-H|DU|^2} holds with $|V|^2=\bar{w}$.
This defines ${\cal H}|\nabla U|^2$ as an element of $(W^{1,\infty}(\partial\Om))'$, 
the dual space of $W^{1,\infty}(\partial\Om)$, 
\end{remark}
}

{{We conclude with an interesting consequence of Theorem \ref{th:e_f''(0)=}.
A priori we expect a continuity property wrt $V$ for the norm of differentiability which is $W^{1,\infty}$. But one can actually get the following improved continuity property.}
\begin{corollary}\label{c:e',e''->H^1}
For $v=(v_1,v_2)\in W^{1,\infty}(\partial\Om;\R^2)$ let $V=(\bar{v}_1,\bar{v}_2)$ be the extension of $v$ as given by Lemma \ref{l:Hv}.
Then, under the same assumptions as in theorem \ref{th:e_f''(0)=} we have:
\\
i) 
The map {$v\in W^{1,\infty}(\partial\Om;\R^2)\mapsto E_{f}'(\Om)(V)$
extends continuously in ${L^1(\partial\Om;\R^2)}$ and the extension is given by \eqref{e:e_f'}}.
\\
ii) Similarly, the map
{$v\in W^{1,\infty}(\partial\Om;\R^2)\mapsto E_{f}''(\Om)(V,V)$
extends continuously in ${H^1}(\partial\Om;\R^2)\cap L^\infty(\partial\Om;\R^2)$ 
and the extension is given by {\eqref{eq:Ef''}}.}
\end{corollary}
\begin{proof}
First we note that the extension $H$ of Lemma \ref{l:Hv} extends continuously in $L^1(\partial\Om)$ and $H^1(\partial\Om)$.
Then for i) we note that \eqref{e:e_f'} is continuous wrt $V\in L^1(\partial\Om;\R^2)$, so the claim follows from the 
continuity of $H$ in $L^1$.
For the claim ii), first  we note that all the terms of {\eqref{eq:Ef''}} except $\int_{\partial\Om}{\cal H}(\partial_\nu U)^2|V|^2$
are continuous in $H^1(\partial\Om)$ wrt to $V$, so they are continuous wrt to $v$ in $H^1(\partial\Om;\R^2)$ thanks to the
continuity of $H$ in $H^1$. 
For the term  $\int_{\partial\Om}{\cal H}(\partial_\nu U)^2|V|^2$, from \eqref{e:def-H|DU|^2} with $\bar{w}=|V|^2$ 
we see that this term is continuous wrt $V$ in $H^1(\Om;\R^2)\cap L^\infty(\Om;\R^2)$ and we conclude using again the continuity of $H$ in $H^1\cap L^\infty$.
\end{proof}
}

\begin{lemma}\label{l:seminorms,equivalence}
The equivalence of seminorms \eqref{e:|h|H^1/2equiv|g|H1/2} holds.
\end{lemma}
\begin{proof}
It is classical, see for example \cite{DD2012}, that for 
$h\in H^{1/2}(\partial\Omega)$, its seminorm is given by
\begin{equation}
|h|^2_{H^{1/2}(\partial B)}
=
\int_{\partial B}ds_y\int_{\partial B}\frac{|h(x)-h(y)|^2}{|x-y|^2}ds_x.
\label{e:|u|H1/2,1}
\end{equation}
For $\theta,\, \eta\in(-\pi,\pi)$, $x=(\cos\theta,\sin\theta)$, $y=(\cos\eta,\sin\eta)$  we set (with a slight abuse of notations)
 $h(\theta)=h(x)$, $h(\eta)=h(y)$. Then \eqref{e:|u|H1/2,1} is equal to
\begin{eqnarray}
|h|^2_{H^{1/2}(\partial B)}
&=&
\int_{-\pi}^\pi d\eta 
\int_{-\pi}^\pi  
\frac{|h(\theta)-h(\eta))|^2}{4\sin^2\left(\frac{\theta-\eta}{2}\right)}
d\theta
=
I_1+I_2+I_3,\;\; \mbox{where} 
\label{e:|u|H1/2,1'}\\
I_i
&=&
\iint_{A_i}\frac{|h(\theta)-h(\eta)|^2}{4\sin^2\left(\frac{\theta-\eta}{2}\right)}d\theta d\eta,\;\, i=1,2,3,
\nonumber
\end{eqnarray}
and
\[
\hspace*{-2mm}
\begin{array}{l}
A_1=
\{(\theta,\eta)\in (-\pi,\pi)^2;\, -\pi<\theta-\eta<\pi\},
\\
A_2=
\{(\theta,\eta)\in(-\pi,\pi)^2;\, \pi<\theta-\eta<2\pi\}
=
\{(\theta,\eta)\in(0,\pi)\times(-\pi,0);\, \pi<\theta-\eta<2\pi\},
\\
A_3=
\{(\theta,\eta)\in(-\pi,\pi)^2;\, -2\pi<\theta-\eta <-\pi\}
=
\{(\theta,\eta)\in(-\pi,0)\times(0,\pi);\, -2\pi<\theta-\eta < -\pi\}.
\end{array}
\]
Let $\hat{\theta}=\theta-2\pi$, $\hat{\eta}=\eta$, $\hat{A}_2=A_2-(2\pi,0)$. Then
\begin{eqnarray}
I_2
&=&
\iint_{A_2}\frac{|h(\theta)-h(\eta))|^2}{4\sin^2\left(\frac{\theta-\eta}{2}\right)}d\theta d\eta
=
\iint_{A_2}\frac{|h(\hat{\theta}+2\pi)-h(\hat{\eta})|^2}{4\sin^2\left(\frac{\hat{\theta}-\hat{\eta}}{2}+\pi\right)}d\theta d\eta
=
\iint_{\hat{A}_2}\frac{|h(\hat{\theta})-h(\hat{\eta})|^2}{4\sin^2\left(\frac{\hat{\theta}-\hat{\eta}}{2}\right)}d\hat{\theta} d\hat{\eta}.
\nonumber
\end{eqnarray}
One can check that 
\[
\hat{A}_2=
\{(\hat{\theta},\hat{\eta})\in (-2\pi,-\pi)\times(-\pi,0);\, 
-\pi<\hat{\theta}-\hat{\eta}<0\}.
\]
Similarly, let $\hat{\theta}=\theta+2\pi$, $\hat{\eta}=\eta$, 
$\hat{A}_3=A_3+(2\pi,0)$. Then 
\begin{eqnarray}
I_2
&=&
\iint_{A_3}\frac{|h(\theta)-h(\eta))|^2}{4\sin^2\left(\frac{\theta-\eta}{2}\right)}d\theta d\eta
=
\iint_{A_3}\frac{|h(\hat{\theta}-2\pi)-h(\hat{\eta})|^2}{4\sin^2\left(\frac{\hat{\theta}-\hat{\eta}}{2}-\pi\right)}d\theta d\eta
=
\iint_{\hat{A}_3}\frac{|h(\hat{\theta})-h(\hat{\eta})|^2}{4\sin^2\left(\frac{\hat{\theta}-\hat{\eta}}{2}\right)}d\hat{\theta} d\hat{\eta}.
\nonumber
\end{eqnarray}
Similarly to $\hat{A}_2$, one can check that 
\[
\hat{A}_3=
\{(\hat{\theta},\hat{\eta})\in (\pi,2\pi)\times(0,\pi);\,
0<\hat{\theta}-\hat{\eta}<\pi\}.
\]
Therefore we get
\[
|h|^2_{H^{1/2}(\partial B)}
=
I_1+I_2+I_3
=
\iint_{\hat{A}_1\cup A_2\cup \hat{A}_3}
\frac{|h(\hat{\theta})-h(\hat{\eta})|^2}{4\sin^2\left(\frac{\hat{\theta}-\hat{\eta}}{2}\right)}d\hat{\theta} d\hat{\eta}
=
\int_{-\pi}^\pi
d{\eta}
\int_{\{|\theta-\eta|<\pi\}}
\frac{|h({\theta})-h({\eta})|^2}{4\sin^2\left(\frac{{\theta}-{\eta}}{2}\right)}d{\theta},
\]
because $A_1\cup\hat{A}_2\cup\hat{A}_3$ is the parallelogram
$\{(\theta,\eta),\, \eta\in(-\pi,\pi),\; |\theta-\eta|<\pi\}$.
Note that we have
\[
\frac{1}{\pi}|\theta-\eta|
\leq
\sin\left|\frac{\theta-\eta}{2}\right|\leq\frac{|\theta-\eta|}{2},\quad
\forall (\theta,\eta)\in A_1\cup\hat{A}_2\cup\hat{A}_3,
\]
which combined with the last equality proves the claim.
\end{proof}

\noindent{\bf Acknowledgements.} This work was partially supported by the project ANR-18-CE40-0013 SHAPO financed by the French Agence Nationale de la Recherche (ANR). Arian Novruzi acknowledges the support of the Natural Sciences and Engineering Research Council of Canada (NSERC)/
remercie le Conseil de recherches en sciences naturelles et en g\'enie du Canada (CRSNG) de son soutien.

\bibliographystyle{plain}
\bibliography{references-4}

\end{document}